\documentclass[preprint,12pt]{elsarticle}

\usepackage{enumerate}
\usepackage{amssymb}
\usepackage{lipsum}
\usepackage[a4paper, total={7.0in, 9.8in}]{geometry}
\usepackage{booktabs}
\usepackage{amsmath,amssymb,url}
\usepackage{enumitem} 
\usepackage{graphics} 
\usepackage{array}
\usepackage{dsfont}
\usepackage[all]{xy}
\usepackage{amsthm}
\usepackage{tikz}
\usepackage{mathrsfs}

\numberwithin{equation}{section}
\newtheorem{theorem}{Theorem}[section]
\newtheorem{lemma}[theorem]{Lemma}

\newcommand{\dn}{\mathord{\downarrow}\hspace{0.05em}}
\newcommand{\up}{\mathord{\uparrow}\hspace{0.05em}}

\newcommand\blfootnote[1]{%
\begingroup
\renewcommand\thefootnote{}\footnote{#1}%
\addtocounter{footnote}{-1}%
\endgroup
}

\journal{}

\begin{document}

\begin{frontmatter}



\title{The product of two Scott sober complete lattices is not always Scott sober}


\author{Zhengmao He}
\address{School of Sciences, Southwest Petroleum University, Chengdu 610500, Sichuan, China}
\begin{abstract} In this paper, we obtain two complete lattices that are sober with respect to the Scott topology, whereas the Scott space of their product is not a sober space. Thus, we answer a question posed by Miao, Xi, Jia, Li and Zhao in \cite{EEF31}.
\end{abstract}
\begin{keyword} Scott topology; sober space; complete lattice\\
\vspace*{0.2cm}
\end{keyword}


\end{frontmatter}
\blfootnote{This work is supported by the National Natural Science Foundation of China (Grant nos.12471438,12601906) and the Sichuan Science and Technology Program (Grant No. 2026NSFSC0784).}
\blfootnote{E-mail address: hezhengmaomath@163.com (Z.M.He).}


\section{Introduction}
Sobriety is a fundamental property in general topology and domain theory. The representation of topological spaces by their lattices of open sets plays a fundamental role in Stone duality. In its classical form, Stone duality establishes a dual equivalence between Boolean algebras and compact zero dimensional Hausdorff spaces, whereas the dual objects are the spectral spaces in the distributive lattice setting. At the level of frames, the correspondence between topological spaces and lattices of open sets restricts to a dual equivalence between sober spaces and spatial frames. In commutative algebra, sobriety arises naturally in the prime spectra of commutative rings. Hochster's characterization of prime spectra states that the spaces homeomorphic to the spectra of commutative rings are precisely the quasi-compact sober spaces that admit a basis of quasi-compact open sets closed under finite intersections(\cite{R4}). A further source of sober spaces is domain theory, developed by Scott to provide mathematical structures for denotational semantics(\cite{R90}). The Scott topology of a poset is one of the most important order topologies in domain theory. A center problem is when the Scott topology of a dcpo is sober(\cite{GG03}). A classical result is that every continuous dcpo and more generally quasicontinuous dcpo is sober with respect to the Scott topology (\cite{GG03}). In addition, the Hofmann-Mislove theorem provides important link. In a sober space, nonempty compact saturated subsets correspond to proper Scott open filters of the open-set lattice (\cite{GG03}).

As we know, a poset equipped with the Scott topology being sober must become a directed complete poset. There are a rich source of examples that directed completeness alone does not imply Scott sobriety. Johnstone first constructed a countable dcpo whose Scott space is not sober (\cite{U14}) and Isbell subsequently constructed a uncountable complete lattice with a non-sober Scott space (\cite{U13}). More recent counterexamples show that even strong order conditions are in general not sufficient for a dcpo to be sober in its Scott topology. Xu, Xi and Zhao produced a complete Heyting algebra whose Scott space is not sober(\cite{EEE20}), while Miao, Xi, Li and Zhao constructed a countable distributive complete lattice with a non sober Scott space(\cite{AB75}). Further recent preprints provide natural examples of Scott non sobriety. Xi, Shen, and Zhao prove that the complete Boolean algebra of regular open subsets of the real line is not Scott sober(\cite{EEF32}). In a subsequent paper, they characterize Scott sobriety of regular open algebras of second-countable $T_3$ spaces by density of the isolated points(\cite{R2}). He, Dong, and Wang prove that countable dcpos with core-compact Scott spaces are Scott sober, while the open set lattice of the rational number space has a non sober Scott topology (\cite{R66}). These results separate Scott sobriety from completeness, finite distributivity, countability and Boolean algebras.

In addition to the sobriety of the Scott topology on an individual dcpo, determining when the Scott topology on the product of two dcpos is sober is also a difficult problem. Ordinary topological products of sober spaces are sober(\cite{GG03}). However, for two dcpos $P$ and $Q$,
$\Sigma(P\times Q)$ and $\Sigma P\times\Sigma Q$ need not coincide. One always has
$$
\sigma(P)\times\sigma(Q)\subseteq\sigma(P\times Q),
$$
where the left-hand side denotes the product topology. Equality holds for continuous dcpos, so their product sobriety follows from the classical theory. For an arbitrary dcpo whose Scott topology is sober, the inclusion considered above may be strict. As a result, the sobriety of their product of these dcpos does not follow immediately. The general product problem was recorded by Xu and Zhao as Question 4.6 in their survey of open problems on well-filtered and sober spaces.(\cite{T2}). Specifically, this question is whether the product of two Scott sober dcpos must again be Scott sober. Miao, Xi, Jia, Li and Zhao answered this question negatively and specifically, they constructed two Scott sober dcpos whose product is not Scott sober(\cite{EEF31}). Their examples are not complete lattices and the complete lattice case was left unresolved in their paper.(\cite{EEF31}). They also obtained a positive result for products of Scott open set lattices of countable posets(\cite{EEF31}). This leads to the question addressed here:
$$
\text{If A and B are complete lattices such that}\
\Sigma A\text{ and }\Sigma B\text{ are sober,}\
\text{must }\Sigma(A\times B)\text{ be sober?}
$$
Recent preprints have further identified the additional hypotheses that preserve the positive results for products. He proves that $\Sigma P\times \Sigma Q=\Sigma(P\times Q)$ if $\Sigma P\times \Sigma Q$ is Fr\'{e}chet-Urysohn(\cite{R6}). Xu proves that the equality holds for arbitrary products of countably presented frames(\cite{R1}). Later, Xu and Ji extend the corresponding conclusions to arbitrary products of countable meet-continuous complete lattices(\cite{EAS1}). These results cover a wide range of situations in which weak first-countability or countability assumptions guarantee that the Scott topology on a product coincides with the product of the Scott topologies. In this paper, we give a negative answer to the complete lattice product question. Concretely, there exist complete lattices $A$ and $B$ such that
$$
\Sigma A\text{ and }\Sigma B\text{ are sober but}\
\Sigma(A\times B)\text{ is not sober}.
$$
Moreover, the factors $A$ and $B$ can be chosen to be co-frames. Thus, $\Sigma(A\times B)$ does not coincide with $\Sigma A\times\Sigma B$. Our construction uses the dcpos $P_{1}$ and $P_{2}$ constructed by Miao, Xi, Jia, Li and Zhao (\cite{EEF31}) and passes to their Scott closed set lattices
$$
A=\Gamma P_1,\ B=\Gamma P_2.$$
Through a careful analysis of the order relations on $P_{1}$ and $P_{2}$, we establish a series of lemmas relating the Scott topology sobriety of the sublattices of $\Gamma P_{1}$ and $\Gamma P_{2}$. Using these lemmas and a pair of Scott continuous maps between these lattices, we prove that both$\Gamma P_{1}$ and $\Gamma P_{2}$ are sober in their Scott topology. Using the Scott irreducible closed sets on $P_{1}\times P_{2}$ showed in \cite{EEF31} and the principal ideal map given below
$$
P_1\times P_2\longrightarrow\Gamma P_1\times\Gamma P_2,
\qquad
(x,y)\longmapsto
\bigl(\downarrow_{P_1}x,\downarrow_{P_2}y\bigr),
$$
the Scott closure of its image yields an irreducible closed subset without a greatest element.
This conclusion shows that Scott sobriety is not preserved by the binary products even within the class of co-frames.
\section{Preliminaries}

\quad Let $P$ be a partially ordered set (poset, for short). Given $A\subseteq P$, we write $\up A$ for the set $\{x\in P\mid \exists\ a\in A, a\leq x \}$ and $\dn A$ for the set $\{x\in P\mid \exists\ a\in A, x\leq a\}$. A nonempty subset $D\subseteq P$ is {\em directed} if for all $a,b\in D$, there is $c\in D$ with $a,b\leq c$. Dually, a nonempty subset $D\subseteq P$ is {\em filtered} if $D$ is directed in the dual poset $P^{op}$. $P$ is called a {\em directed complete poset} ({\em dcpo}, for short) if every directed subset of $P$ has a supremum. An element $a$ is {\em maximal} in $A\subseteq P$ if $a\in A$ and for each $x\in A$, $a\leq x$ implies $a=x$. For $x,y\in P$, we say that $x$ is {\em way below} $y$, in symbols $x\ll y$, if for each directed subset $D\subseteq P$ for which $\bigvee D$ exists, $y\leq\bigvee D$ implies $\up x\cap D\neq\emptyset$. In particular, $x$ is called a {\em compact element} if $x\ll x$. A {\em complete lattice} is a poset in which every subset has a supremum and an infimum.

\quad Let $P$ be a poset. A subset $U\subseteq P$ is called {\em Scott open} if
$U=\up U$ (or equivalently $U$ is upper) and for each directed subset $D$ with $\bigvee D$ existing, $\bigvee D\in U$ implies
$D\cap U\neq \emptyset$. All Scott open sets of $P$ form the Scott topology $\sigma(P)$. We write $\Sigma P$ for $(P,\sigma(P))$ and $\Gamma P$ for the family of all closed sets (also called Scott closed sets) in $\Sigma P$. The upper topology $\upsilon(P)$ of $P$ is generated by the collection $\{P\setminus\dn x\mid x\in P\}$. Let $P,Q$ be dcpos. Then the mapping $f:\Sigma P\longrightarrow \Sigma Q$ is continuous (or called {\em Scott continuous}) if and only if for each directed set $D\subseteq P$, $f(\bigvee D)=\bigvee\limits_{d\in D}f(d)$.

\quad Let $X$ be a $T_{0}$ space and $A\subseteq X$. We write $\overline{A}_{X}$ for the closure of $A$ in $X$. The {\em specialization order} $\leq$ on $X$ is defined by
$$x\leq y \Longleftrightarrow x\in \overline{\{y\}}_{X}.$$
Unless otherwise stated, throughout
the paper, whenever an order-theoretic concept is mentioned in the context of a $T_{0}$ space $X$, it is to be interpreted with respect to the specialization order on $X$. Under the specialization order, we can conclude that for each $x\in X$, $\overline{\{x\}}_{X}=\dn x$. A nonempty subset $F\subseteq X$ is called {\em irreducible} if for each pair of closed sets $A, B\subseteq X$, $F\subseteq A\cup B$ implies $F\subseteq A$ or $F\subseteq B$. A $T_{0}$ space $X$ is {\em sober} if every irreducible closed set $A$ is equal to $\overline{\{x\}}_{X}$ for some $x\in X$. Equivalently, a $T_{0}$ space $X$ is sober if and only if each irreducible closed subset of $X$ contains a largest element with respect to the specialization order on $X$.

\section{The dcpos $P_{1}$ and $P_{2}$ }

In this section, we recall two dcpos $P_1$ and $P_2$ constructed by Miao, Xi, Jia, Li and Zhao (\cite{EEF31}).

Let $\mathbb{N}$ be the set of natural numbers and $\mathbb{N}^{<\omega}$ the set of all nonempty finite words over $\mathbb{N}$. For $x=x_1\cdots x_n,\ y=y_1\cdots y_m$ in $\mathbb{N}^{<\omega}$, define the order $\preceq$ by
$$ x\preceq y
 \Longleftrightarrow
  n\leq m\text{ and }x_k=y_k\text{ for each }1\leq k\leq n.
$$
 If $s\in\mathbb{N}^{<\omega}$ and $j\in\mathbb{N}$, we write $s\cdot j$ for the word obtained by adding $j$ to $s$.

Let
\[
  M=\mathbb{N}\sqcup\mathbb{N}^{<\omega},
\]
ordered by
\[x\leq_M y
\Longleftrightarrow
x\preceq y, x,y\in\mathbb{N}^{<\omega}\ \text{or}\
x\leq y, x,y\in\mathbb{N}.
\]
Take
\[
  \triangle=\{(a,b)\in\mathbb{N}^2:a<b\}
\]
and
\[
  L=(\triangle\times M)\sqcup\{\top\}.
\]
For each $(a,b)\in\Delta$ and $x\in M$, write
\[
  x_{a,b}=((a,b),x).
\]
Define the order $\leq_L$ on $L$ by
\[
  x_{a,b}\leq_L y_{a',b'}
\Longleftrightarrow
  (a,b)=(a',b')\text{ and }x\leq_My,
\]
and for each $u\in L$, $u\leq_L\top$.

In \cite{EEF31}, they construct an injection $i:\triangle\longrightarrow\mathcal P(\mathbb{N})$
such that
\[
  i(a,b)\cap i(a',b')=\varnothing
  \quad\text{if }(a,b)\neq(a',b')
\]
and further
\[
  k>b\ \text{for every }(a,b)\in\triangle\text{ and }k\in i(a,b).
\]
For each $(a,b)\in\Delta$, there exists an order preserving injection
\[
  f_{a,b}:\mathbb{N}^{<\omega}\longrightarrow i(a,b).
\]

Let $B=\mathbb{N}\times\mathbb{N}\times L$ and write $T_{m,n}=(m,n,\top)$ for each pair $m,n\in\mathbb{N}$.

Define the following four relations $\sqsubset_1,\sqsubset_2,\sqsubset_3,\sqsubset_4$ on $B$ by
\begin{align}
(m,n,z)&\sqsubset_1(m,n,y)
  &&\text{if }z<_Ly,\\
(a,n,x_{a,b})&\sqsubset_2T_{f_{a,b}(x),n+1}
  &&\text{if }(a,b)\in\triangle,\ x\in\mathbb{N}^{<\omega},\\
(b,n,k_{a,b})&\sqsubset_3T_{f_{a,b}(k),n+1}
  &&\text{if }(a,b)\in\triangle,\ k\in\mathbb{N},\\
(f_{a,b}(s),n,k_{a,b})&\sqsubset_4T_{f_{a,b}(s\cdot k),n}
  &&\text{if }(a,b)\in\triangle,\ s\in\mathbb{N}^{<\omega},\ k\in\mathbb{N}.
\end{align}
For relations $R$ and $S$ on $B$, define
\[
  x\,(R\mathbin{;}S)\,z
  \quad\Longleftrightarrow\quad
  \text{there exists }y\in B\text{ such that }x\,R\,y\text{ and }y\,S\,z.
\]
Define
\[\sqsubset=\sqsubset_1\cup\sqsubset_2\cup\sqsubset_3\cup\sqsubset_4
\cup(\sqsubset_1\mathbin{;}\sqsubset_2)
  \cup(\sqsubset_1\mathbin{;}\sqsubset_3)
  \cup(\sqsubset_1\mathbin{;}\sqsubset_4).\]
  
The order on $B$ is $\sqsubseteq=\sqsubset\cup\operatorname{id}_B$.

Then the maximal elements of $B$ are ${\rm Max} B=\{T_{m,n}\mid m,n\in\mathbb{N}\}$.

The posets $P_1$ and $P_2$ are obtained by attaching uncountably many chains to the common subposet $B$.

Let $\mathbb{N}^{\mathbb{N}}$ be the set of all mappings from $\mathbb{N}$ to $\mathbb{N}$.

For $h,k\in\mathbb{N}^{\mathbb{N}}$ and $m,n\in\mathbb{N}$, define
\[
  (h,n)\leq(k,m)
\Longleftrightarrow
  h=k\text{ and }n\leq m.
\]

Let $P_1=(\mathbb{N}^{\mathbb{N}}\times\mathbb{N})\sqcup B\sqcup\{\top_1\}$. The order on $P_1$ is the smallest partial order containing the orders on $\mathbb{N}^{\mathbb{N}}\times\mathbb{N}$ and $B$ which satisfies
\[
  (h,n)<T_{h(n),n}\ \text{and}\ u<\top_1
\]
for all $(h,n)\in\mathbb{N}^{\mathbb{N}}\times\mathbb{N}$ and $u\in P_1\setminus\{\top_1\}$.

Fix a bijection $\varphi:\mathbb{N}\longrightarrow\triangle$
and take $E_n=i(\varphi(n))$, for each $n\in\mathbb{N}$.

Set
\[
  \chi=
  \left\{
    g\mathrel{\Big|}
    g:\mathbb{N}\longrightarrow\bigcup_{n\in\mathbb{N}}E_n,\
    g(n)\in E_n\text{ for every }n\in\mathbb{N}
  \right\}.
\]
For $g,g'\in\chi$ and $m,n,k,p\in\mathbb{N}$, define
\[
  (g,n,k)\leq(g',m,p)
\Longleftrightarrow
  g=g',\ k=p,\text{ and }n\leq m.
\]
Let $P_2=(\chi\times\mathbb{N}\times\mathbb{N})\sqcup B\sqcup\{\top_2\}$.

The order on $P_2$ is the least partial order containing the orders on $\chi\times\mathbb{N}\times\mathbb{N}$ and $B$ which satisfies
\[
  (g,n,k)<T_{g(n),k},
  \qquad
  u<\top_2
\]
for all $(g,n,k)\in\chi\times\mathbb{N}\times\mathbb{N}$ and $u\in P_2\setminus\{\top_2\}$.
\section{Main Results}

In this section, we show that $\Sigma\Gamma P_{1}$ and $\Sigma\Gamma P_{2}$ are sober but $\Sigma (\Gamma P_{1}\times\Gamma P_{2})$ is not sober.

\begin{lemma}{\rm Let $P$ be a dcpo and $V\in\sigma (P)$. Then $\Gamma V=\{C\cap V\mid C\in\Gamma P\}$.}

\end{lemma}

\begin{proof} Let $C\in\Gamma P$. Then $C\cap V$ is lower in the subposet $V$ and closed under directed suprema of $V$. This means that $C\cap V\in\Gamma V$. On the contrary, suppose $W\in\Gamma V$. Equivalently, $W$ is a Scott closed subset of $V$. Then $V\setminus W$ is Scott open in $V$. It follows from $V\in\sigma(P)$ that $V\setminus W$ is Scott open in $P$. Set $$Q=P\setminus(V\setminus W)=(P\setminus V)\cup W.$$ Whence $Q\in\Gamma(P)$ and $W=Q\cap V$. Therefore, the Scott closed sets of the subposet $V$ are exactly all sets $C\cap V$. \end{proof}

\begin{lemma} {\rm Let $U$ be a dcpo and $V\in\sigma(U)$. If $$N(U)=\{\bigvee D\mid D\subseteq U\ \mbox{is directed and}\ \bigvee D\not\in D\}\subseteq V,$$ then $\Sigma \Gamma U$ is sober if and only if $\Sigma\Gamma V$ is sober.}
\end{lemma}

\begin{proof} Define two mappings $q:\Gamma U\longrightarrow\Gamma V$ and $e:\Gamma V\longrightarrow\Gamma U$ by
$$q(F)=F\cap V,\ \ e(G)=\overline{G}_{\Sigma U}.$$

$\mathbf{Claim} \ 1$: $q$ and $e$ are Scott continuous.

Suppose $\mathscr{C}\subseteq\Gamma U$ is directed. Then $$q(\bigvee\limits_{\Gamma U}\mathscr{C})=\overline{(\bigcup\mathscr{C})}_{\Sigma U}\cap V$$ and $$\bigvee\limits_{C\in\mathscr{C}}\limits^{\Gamma V}q(C)=\bigvee\limits_{C\in\mathscr{C}}\limits^{\Gamma V}(C\cap V)=\overline{((\bigcup\mathscr{C})\cap V)}_{\Sigma V}.$$
Clearly, $\overline{(\bigcup\mathscr{C})}_{\Sigma U}\cap V$ is an upper bound of $\{C\cap V\mid C\in\mathscr{C}\}$ in $\Gamma V$. Fix an $x\in\overline{(\bigcup\mathscr{C})}_{\Sigma U}\cap V$ and a $W\in\sigma (V)$ being a Scott open neighbourhood of $x$. As $V\in\sigma(U)$, $W\in\sigma(U)$. It follows from $x\in\overline{(\bigcup\mathscr{C})}_{\Sigma U}$ that $W\cap(
\bigcup\mathscr{C})\neq\emptyset$. As a consequence, $x\in\overline{((\bigcup\mathscr{C})\cap V)}_{\Sigma V}$. So $$\overline{(\bigcup\mathscr{C})}_{\Sigma U}\cap V=\overline{((\bigcup\mathscr{C})\cap V)}_{\Sigma V}$$ and thus $q$ is Scott continuous.

Let $\mathscr{D}\subseteq \Gamma V$ be a directed subfamily. Then $$e(\bigvee\limits_{\Gamma V}\mathscr{D})=e(\overline{(\bigcup\mathscr{D})}_{\Sigma V})=\overline{(\overline{(\bigcup\mathscr{D})}_{\Sigma V})}_{\Sigma U}$$ and
$$\bigvee\limits_{D\in\mathscr{D}}\limits^{\Gamma U} e(D)=\bigvee\limits_{D\in\mathscr{D}}\limits^{\Gamma U} \overline{D}_{\Sigma U}=\overline{(\bigcup\limits_{D\in\mathscr{D}}\overline{D}_{\Sigma U})}_{\Sigma U}.$$
Clearly, $\bigvee\limits_{D\in\mathscr{D}}\limits^{\Gamma U} e(D)\subseteq e(\bigvee\limits_{\Gamma V}\mathscr{D})$. Choose a $y\in\overline{(\overline{(\bigcup\mathscr{D})}_{\Sigma V})}_{\Sigma U}$ and a $S\in\sigma(U)$ with $y\in S$. Then we have $S\cap\overline{(\bigcup\mathscr{D})}_{\Sigma V}\neq\emptyset$. Select a $z\in S\cap\overline{(\bigcup\mathscr{D})}_{\Sigma V}$. Now, $S\cap V$ is a Scott open neighbourhood of $z$ in $\Sigma V$. By $z\in\overline{(\bigcup\mathscr{D})}_{\Sigma V}$, $S\cap V\cap(\bigcup\mathscr{D})\neq\emptyset$. This leads to $S\cap(\bigcup\limits_{D\in\mathscr{D}}\overline{D}_{\Sigma U})\neq\emptyset$. Consequently, $y\in\overline{(\bigcup\limits_{D\in\mathscr{D}}\overline{D}_{\Sigma U})}_{\Sigma U}$ and hence $\bigvee\limits_{D\in\mathscr{D}}\limits^{\Gamma U} e(D)=e(\bigvee\limits_{\Gamma V}\mathscr{D})$. Therefore, $e$ is also Scott continuous.

$\mathbf{Claim} \ 2$: $q\circ e=id_{\Gamma V}$ and $e\circ q\leq id_{\Gamma U}$.

Let $G^{\ast}\in\Gamma V$. We have $$G^{\ast}\subseteq q\circ e(G^{\ast})=\overline{G^{\ast}}_{\Sigma U}\cap V.$$ Assume there is a $g\in\overline{G^{\ast}}_{\Sigma U}\cap V$ but $g\not\in G^{\ast}$. Then $V\setminus G^{\ast}$ is a Scott open set containing $g$ in $\Sigma V$. By $V\in\sigma(U)$, $V\setminus G^{\ast}\in\sigma(U)$. We obtain that $(V\setminus G^{\ast})\cap G^{\ast}\neq\emptyset$ because $g\in\overline{G^{\ast}}_{\Sigma U}$, a contradiction. So we can conclude that $G^{\ast}=q\circ e(G^{\ast})$. Hence $q\circ e=id_{\Gamma V}$.

Let $F^{\ast}\in\Gamma U$. Then $$e\circ q(F^{\ast})=\overline{(F^{\ast}\cap V)}_{\Sigma U}\subseteq\overline{F^{\ast}}_{\Sigma U}=F^{\ast}.$$ Thus $e\circ q\leq id_{\Gamma U}$.

$\mathbf{Claim} \ 3$: If $\Sigma \Gamma U$ is sober, then $\Sigma \Gamma V$ is sober.

Let $\mathscr{A}\subseteq\Gamma V$ be an irreducible closed subset in $\Sigma \Gamma V$. Since $e$ is Scott continuous, $\overline{e(\mathscr{A})}_{\Sigma \Gamma U}$ is also irreducible in $\Sigma \Gamma U$. By the sobriety of $\Sigma \Gamma U$, $$\overline{e(\mathscr{A})}_{\Sigma \Gamma U}=\overline{\{F^{\star}\}}_{\Sigma \Gamma U}$$ for some $F^{\star}\in\Gamma U$. As $q$ is Scott continuous and $q\circ e=id_{\Gamma V}$, $q^{-1}(\mathscr{A})$ is Scott closed in $\Sigma \Gamma U$ and $e(\mathscr{A})\subseteq q^{-1}(\mathscr{A})$. This means that $\overline{e(\mathscr{A})}_{\Sigma \Gamma U}\subseteq q^{-1}(\mathscr{A})$. So we know $F^{\star}\in q^{-1}(\mathscr{A})$. Equivalently, $q(F^{\star})\in\mathscr{A}$. In addition, $e(A)\subseteq F^{\star}$ for each $A\in\mathscr{A}$. Then we get $$q(e(A))=A\subseteq q(F^{\star}).$$ In other words, $q(F^{\star})$ is the greatest element in $\mathscr{A}$. This implies that $\mathscr{A}=\overline{\{q(F^{\star})\}}_{\Sigma\Gamma V}$. Hence, $\Sigma\Gamma V$ is sober.

$\mathbf{Claim} \ 4$: If $\Sigma \Gamma V$ is sober, then $\Sigma \Gamma U$ is sober.

Let $\mathscr{B}\subseteq\Gamma U$ be an irreducible Scott closed set. Set $$S=\bigcup\mathscr{B}.$$

$\mathbf{Subclaim} \ 4.1$: $S$ is Scott closed in $\Sigma U$.

Indeed, the mapping $\xi:U\longrightarrow \Gamma U$ is Scott continuous, where $\xi$ is defined by $\xi(u)=\dn u$. This leads to $\xi^{-1}(\mathscr{B})$ is Scott closed in $\Sigma U$. Furthermore, it is not difficult to verify that $S=\xi^{-1}(\mathscr{B})$. Thus, $S\in\Gamma U$.

Take $$H=S\cap V.$$ As $q$ is Scott continuous, $\mathscr{E}=\overline{q(\mathscr{B})}_{\Sigma\Gamma V}$ is an irreducible closed subset. Then $\mathscr{E}$ contains a largest element $G^{\star}$ since $\Sigma \Gamma V$ is sober.

$\mathbf{Subclaim} \ 4.2$: $H=G^{\star}$.

Choose a $x\in H$. Then $x\in I$ for some $I\in\mathscr{B}$. Now, we have that $$x\in I\cap V=q(I)\in\mathscr{E}\subseteq \dn_{\Gamma V}G^{\star}.$$ Whence $x\in G^{\star}$. This proves $H\subseteq G^{\star}$.

By Lemma 4.1, $H$ is Scott closed in subposet $V$. Clearly, $q(\mathscr{B})\subseteq\dn_{\Gamma V}H$ and $\dn_{\Gamma V}H$ is Scott closed in $\Sigma\Gamma V$. So we have $\mathscr{E}=\overline{q(\mathscr{B})}_{\Sigma\Gamma V}\subseteq\dn_{\Gamma V}H$. In particular, $G^{\star}\in\dn_{\Gamma V}H$, that is $G^{\star}\subseteq H$.

Therefore, $H=G^{\star}$.

Take $$F_{0}=e(H)=\overline{(S\cap V)}_{\Sigma U}.$$

$\mathbf{Subclaim} \ 4.3$: $F_{0}\in\mathscr{B}$.

For each $J\in\mathscr{B}$, since $e\circ q(J)\subseteq J$ and $\mathscr{B}$ is lower in $\Gamma U$, $e\circ q(J)\in \mathscr{B}$. It follows that $q(\mathscr{B})\subseteq e^{-1}(\mathscr{B})$. Furthermore, we obtain $\mathscr{E}=\overline{q(\mathscr{B})}_{\Sigma\Gamma V}\subseteq e^{-1}(\mathscr{B})$ as $e^{-1}(\mathscr{B})\in\Gamma(\Gamma V)$. In particular, $G^{\star}\in e^{-1}(\mathscr{B})$, or equivalently $e(G^{\star})=e(H)=F_{0}\in\mathscr{B}$.

$\mathbf{Subclaim} \ 4.4$: Each $x\in S\setminus F_{0}$ is compact in subposet $S$.

By Subclaim 4.1, we have $$N(S)\subseteq N(U)\cap S\subseteq V\cap S\subseteq F_{0}.$$

Let $s\in S\setminus F_{0}$ and $D^{\ast}\subseteq S$ a directed subset with $s\leq \bigvee\limits_{S} D^{\ast}$. Assume that $\bigvee\limits_{S} D^{\ast}\not\in D^{\ast}$. Then we obtain $\bigvee\limits_{S} D^{\ast}\in F_{0}$. As $F_{0}\in \Gamma U$, $s\in F_{0}$, impossible.
As a consequence, $\bigvee\limits_{S} D^{\ast}\in D^{\ast}$ and then $s\ll_{S} s$.

$\mathbf{Subclaim} \ 4.5$: For each $K\in\mathscr{B}$ and each $t\in S\setminus F_{0}$, $K\cup\dn_{U}t\in\mathscr{B}$.

As $t\ll_{S}t$, $\up_{S}t$ is Scott open in subposet $S$. So $K_{t}=S\setminus\up_{S}t$ is Scott closed in subposet $S$. By $S\in\Gamma U$, $K_{t}$ is Scott closed in $\Sigma U$. Take $$\mathscr{O}_{t}=\Gamma U\setminus\dn_{\Gamma U}K_{t}.$$ Then $\mathscr{O}_{t}$ is Scott open in $\Sigma\Gamma U$. For each $R\in\mathscr{B}$, we have
$$R\in\mathscr{O}_{t}\Longleftrightarrow R\not\subseteq K_{t}\Longleftrightarrow R\cap\up_{S}t\neq\emptyset\Longleftrightarrow t\in R.$$
Set $$\mathscr{R}_{t}=\mathscr{O}_{t}\cap\mathscr{B}.$$
By $t\in S=\bigcup\mathscr{B}$, choose a $J_{t}\in\mathscr{B}$ such that $t\in J_{t}$. Then $J_{t}\in\mathscr{O}_{t}\cap\mathscr{B}$, or equivalently $\mathscr{R}_{t}\neq\emptyset$. Clearly, $\overline{\mathscr{R}_{t}}_{\Sigma\Gamma U}\subseteq\mathscr{B}$. Suppose $J^{\star}\in\mathscr{B}$ and $\mathscr{F}\in\sigma(\Gamma U)$ with $J^{\star}\in\mathscr{F}$. Since $\mathscr{B}$ is irreducible and $J^{\star}\in\mathscr{F}\cap\mathscr{B},\mathscr{O}_{t}\cap\mathscr{B}\neq\emptyset$, we have $\mathscr{B}\cap\mathscr{O}_{t}\cap\mathscr{F}\neq\emptyset$. This entails that $J^{\star}\in\overline{\mathscr{R}_{t}}_{\Sigma\Gamma U}$. Hence, $\overline{\mathscr{R}_{t}}_{\Sigma\Gamma U}=\mathscr{B}$. Define a mapping $r_{t}:\Gamma U\longrightarrow\Gamma U$ by $$r_{t}(L)=L\cup\dn_{U}t.$$ Suppose $\{L_{i}\mid i\in I\}$ is a directed subfamily of $\Gamma U$. Then $$r_{t}(\bigvee\limits_{i\in I}\limits^{\Gamma U}L_{i})=\dn_{U}t\cup\overline{(\bigcup\{L_{i}\mid i\in I\})}_{\Sigma\Gamma U}=\overline{(\dn_{U}t\cup\bigcup\{L_{i}\mid i\in I\})}_{\Sigma\Gamma U}$$ and
$$\bigvee\limits_{i\in I}\limits^{\Gamma U}r_{t}(L_{i})=\overline{(\bigcup\{L_{i}\cup\dn_{U}t\mid i\in I\})}_{\Sigma\Gamma U}=\overline{(\dn_{U}t\cup\bigcup\{L_{i}\mid i\in I\})}_{\Sigma\Gamma U}.$$ This yields that $$r_{t}(\bigvee\limits_{i\in I}\limits^{\Gamma U}L_{i})=\bigvee\limits_{i\in I}\limits^{\Gamma U}r_{t}(L_{i}).$$ As a result, $r_{t}$ is Scott continuous. For each $R^{\ast}\in\mathscr{R}_{t}$, $t\in R^{\ast}$ and $R^{\ast}\in\mathscr{B}$. So we have $$r_{t}(R^{\ast})=R^{\ast}\cup\dn_{U}t=R^{\ast}\in\mathscr{B}.$$ This leads to $\mathscr{R}_{t}\subseteq r_{t}^{-1}(\mathscr{B})$. By the Scott continuity of $r_{t}$, $r_{t}^{-1}(\mathscr{B})$ is Scott closed on $\Sigma\Gamma U$. Consequently, $$\mathscr{B}=\overline{\mathscr{R}_{t}}_{\Sigma\Gamma U}\subseteq r_{t}^{-1}(\mathscr{B}).$$ In other words, $\forall\ K\in\mathscr{B}$, $K\cup\dn_{U}t\in\mathscr{B}$.

$\mathbf{Subclaim} \ 4.6$: $S\in\mathscr{B}$.
 For each finite subset $T\subseteq S\setminus F_{0}$, let $$Q_{T}=F_{0}\cup(\bigcup\{\dn_{U}t\mid t\in T\}).$$ By Subclaim 4.5, $Q_{T}\in\mathscr{B}$ for each $T$. Now, the family $$\mathscr{Q}=\{Q_{T}\mid T\subseteq S\setminus F_{0}\ \mbox{ is finite}\ \}$$ is directed and contained in $\mathscr{B}$. Since $\mathscr{B}$ is Scott closed in $\Sigma\Gamma U$ and $S\in \Gamma U$, we have $$\bigvee\limits_{\Gamma U}\mathscr{Q}=\overline{(\bigcup\mathscr{Q})}_{\Sigma U}=\overline{S}_{\Sigma U}=S\in\mathscr{B}.$$
By Subclaim 4.6, we can conclude that $\Sigma \Gamma U$ is sober.
\end{proof}

\begin{lemma} {\rm Let $P$ be a dcpo and $Q=U\cap V$ with $U\in\Gamma P,V\in\sigma(P)$. If $\Sigma \Gamma P$ is sober, then $\Sigma\Gamma Q$ is sober.}
\end{lemma}

\begin{proof} By the Scott closedness of $U$, we obtain $$\Gamma U=\{A\in\Gamma P\mid A\subseteq U\}=\{A\cap U\mid A\in\Gamma P\}.$$
Let $\mathscr{A}\subseteq\Gamma U$ be an irreducible closed set in $\Sigma\Gamma U$. Clearly,
the inclusion mapping $i:\Gamma U\longrightarrow\Gamma P$  is Scott continuous. So $i(\mathscr{A})$ is irreducible in $\Sigma\Gamma P$. As $U\in\Gamma P$ and $\mathscr{A}\in\Gamma(\Gamma U)$, we deduce that $$\mathscr{A}=i(\mathscr{A})\in\Gamma(\Gamma P).$$
Then $\mathscr{A}$ contains a largest element $A\in\mathscr{A}$ because  $\Sigma \Gamma P$ is sober. Thus $\Sigma\Gamma U$ is sober.

As $U\in\Gamma P,V\in\sigma(P)$, we have $Q\in\sigma(U)$. By Claim 1, Claim 2 and Claim 3 in Lemma 4.2, $\Sigma\Gamma Q$ is sober.\end{proof}

\begin{lemma} {\rm (see \cite{AB75}) Let $P,Q$ be two posets. If both $P$ and $Q$ contain countably many ideals, none of which has a largest element, then $\Sigma(P\times Q)=\Sigma P\times\Sigma Q$.}
\end{lemma}

\begin{lemma} {\rm(see \cite{RTW1}) Let $P$ be a dcpo. If for each $n$, $$\Sigma(\prod\limits^{n}P)=\prod\limits^{n}\Sigma P,$$ then $\sigma(\Gamma P)=\upsilon(\Gamma P)$.}
\end{lemma}

\begin{lemma} {\rm (see \cite{EEF22}) Let $P$ be a dcpo. Then $(P,\upsilon(P))$ is sober if and only if for each irreducible subset $A$ in $(P,\upsilon(P))$, $\bigvee\limits_{P}A$ exists.}
\end{lemma}

Fix  $$k\in\mathbb{N},\ (a,b)\in\bigtriangleup \ \text{and}\  r_{0}\in\mathbb{N}.$$
Let $\mathcal{T}_{0}$ be the set obtained by adjoining the empty word to $\mathbb{N}^{<\omega}$. For each $s\in\mathcal{T}_{0}$, we write $m_{s}$ for $f_{a,b}(r_{0}.s)$. Take
$$C_{0}=\{(m_{s},k,z)\mid s\in\mathcal{T}_{0},z\in L\}.$$
Then $C_{0}$ is a subdcpo of $P_{1}$. For each $s\in\mathcal{T}_{0}$, let $t_{s}=T_{m_{s},k}=(m_{s},k,\top)$, $A_{s}=\{(m_{s},k,j_{a,b})\mid j\in\mathbb{N}\}$ and
$$P_{r_{s}}=\{(m_{s},k,j_{c,d})\mid j\in\mathbb{N},(c,d)\in\triangle, (c,d)\neq(a,b)\}\cup\{(m_{s},k,w_{c,d})\mid w\in\mathbb{N}^{<\omega},(c,d)\in\triangle\}.$$ Furthermore, by the order on $P_{1}$, for each $l\in\mathbb{N}$, $$(m_{s},k,j_{a,b})<t_{sl}\Longleftrightarrow j\leq l.$$

\begin{lemma} {\rm Let $E\subseteq C_{0}$ be a directed subset with $\bigvee\limits_{C_{0}}E\not\in E$ and $x\in P_{r_{s}}$ such that $x\leq \bigvee\limits_{C_{0}}E$. Then $\bigvee\limits_{C_{0}}E=t_{s}$.}
\end{lemma}

\begin{proof} As $E\subseteq C_{0}$ is a directed and $\bigvee\limits_{C_{0}}E\not\in E$, we know $\bigvee\limits_{C_{0}}E=t_{s^{\ast}}$ for some $s^{\ast}\in\mathcal{T}_{0}$. By assumption, we just consider the following two cases.

Case 1: $x=(m_{s},k,j_{c,d})$ for some $(c,d)\neq(a,b)$.

Then we have $(m_{s},k,j_{c,d})\leq t_{s^{\ast}}=(f_{a,b}(r_{0}.s^{\ast}),k,\top)$. By the order of $P_{1}$, $$(m_{s},k,j_{c,d})\sqsubset_1 t_{s^{\ast}}=(f_{a,b}(r_{0}\cdot s^{\ast}),k,\top) \ \text{or}\ \ (m_{s},k,j_{c,d})\sqsubset_4 t_{s^{\ast}}=(f_{a,b}(r_{0}\cdot s^{\ast}),k,\top).$$

If $(m_{s},k,j_{c,d})\sqsubset_1 t_{s^{\ast}}=(f_{a,b}(r_{0}\cdot s^{\ast}),k,\top)$, then  $f_{a,b}(r_{0}\cdot s^{\ast})=m_{s}=f_{a,b}(r_{0}\cdot s)$. Thus $s^{\ast}=s$ and hence $\bigvee\limits_{C_{0}}E=t_{s}$.

If $(m_{s},k,j_{c,d})\sqsubset_4 t_{s^{\ast}}=(f_{a,b}(r_{0}\cdot s^{\ast}),k,\top)$, then $(c,d)=(a,b)$, impossible.

Case 2: $x=(m_{s},k,w_{c,d})$ for some $w\in\mathbb{N}^{<\omega}$ and $(c,d)\in\triangle$.

Then $(m_{s},k,w_{c,d})\leq t_{s^{\ast}}=(f_{a,b}(r_{0}\cdot s^{\ast}),k,\top)$. Then by the order on $P_{1}$, it follows that
$$(m_{s},k,w_{c,d})\sqsubset_1 t_{s^{\ast}}=(f_{a,b}(r_{0}\cdot s^{\ast}),k,\top).$$

Similarly, we obtain $m_{s}=f_{a,b}(r_{0}\cdot s^{\ast})$. Therefore, $s=s^{\ast}$ and thus $\bigvee\limits_{C_{0}}E=t_{s}$. \end{proof}

\begin{lemma} {\rm $\Sigma\Gamma C_{0}$ is sober.}
\end{lemma}

\begin{proof} Let $\mathscr{A}\subseteq\Gamma C_{0}$ be an irreducible closed subset in $\Sigma\Gamma C_{0}$. Take $$S=\bigcup\mathscr{A}.$$ By the proof of Subclaim 4.1 in Lemma 4.2 , $S\in\Gamma C_{0}$. Consequently, $\Gamma S=\dn_{\Gamma C_{0}}S$ and hence $\mathscr{A}$ is also an irreducible Scott closed subset in $\Sigma\Gamma S$. Set $$J=\{x\in S\mid F\cup\dn_{S}x\in\mathscr{A}, \ \text{for each}\ F\in\mathscr{A}\}.$$

$\mathbf{Claim} \ 1$: $F\cup J\in\mathscr{A}$, for each $ F\in\mathscr{A}$.

Define a mapping $r_{F}:S\longrightarrow\Gamma S$ by $$\forall\ s\in S,\ r_{F}(s)=F\cup\dn_{S}s.$$
Let $D\subseteq S$ be a directed subset. Then we have $r_{F}(\bigvee\limits_{S}D)=F\cup\dn_{S}\bigvee\limits_{S}D$ and $$\bigvee\limits_{d\in D}\limits^{\Gamma S}r_{F}(d)=\overline{\{F\cup\dn_{S}d\mid d\in D\}}_{\Sigma S}=\overline{(F\cup\dn_{S}D)}_{\Sigma S}=\overline{F}_{\Sigma S}\cup\overline{\dn_{S}D}_{\Sigma S}=F\cup\dn_{S}\bigvee\limits_{S}D.$$ This means that $r_{F}$ is Scott continuous. Thus $$J=\bigcap\limits_{F\in\mathscr{A}}r_{F}^{-1}(\mathscr{A})$$ is Scott closed in the subdcpo $S$.
Suppose $E\subseteq J$ is finite. For convenience, we enumerate $E$ as $\{s_{0},s_{1},\cdot\cdot\cdot,s_{n}\}$. Since each $s_{i}$ ($0\leq i\leq n$) belongs to $J$, we deduce by induction that $F\cup\dn E\in\mathscr{A}$. Now, we know that $$\mathscr{A}_{J}=\{F\cup\dn E\mid E \ \text{is a finite subset of}\ J\}$$ is a directed subset of $\mathscr{A}$. As $\mathscr{A}$ is Scott closed in $\Gamma S$, we get $$F\cup J=\overline{\bigcup\mathscr{A}_{J}}_{\Sigma S}=\bigvee\limits_{\Gamma S}\mathscr{A}_{J}\in\mathscr{A}.$$ In particular, $J\in\mathscr{A}$ by taking $F=\emptyset\in\mathscr{A}$.

Let $W=S\setminus J$. Then $W\in\sigma(S)$ and $W$ is a subdcpo of $S$.
Define two mappings $q_{0}:\Gamma S\longrightarrow\Gamma W$ and $e_{0}:\Gamma W\longrightarrow\Gamma S$ by $$ q_{0}(F)=F\cap W\ \text{and}\ e_{0}(H)=J\cup H.$$

$\mathbf{Claim} \ 2$: $q_{0}$ and $e_{0}$ are Scott continuous.

For each $H\in\Gamma W$, we have $$S\setminus(J\cup H)=W\setminus H\in\sigma(W)\subseteq\sigma(S).$$ This means that $J\cup H$ is Scott closed in $S$. So $e_{0}$ is well-defined. By the proof of the Claim 1 in Lemma 4.2, $q_{0}$ is Scott continuous. Suppose $\{H_{i}\mid i\in I\}\subseteq\Gamma W$ is directed. Then $$e_{0}(\bigvee\limits_{i\in I}\limits^{\Gamma W}H_{i})=e_{0}(\overline{\bigcup\{H_{i}\mid i\in I\}}_{\Sigma W})=J\cup\overline{\bigcup\{H_{i}\mid i\in I\}}_{\Sigma W}$$ and
$$\bigvee\limits_{i\in I}\limits^{\Gamma S}e_{0}(H_{i})=\overline{\bigcup\{J\cup H_{i}\mid i\in I\}}_{\Sigma S}=J\cup\overline{\bigcup\{H_{i}\mid i\in I\}}_{\Sigma S}.$$ Clearly, $\bigvee\limits_{i\in I}\limits^{\Gamma S}e_{0}(H_{i})\subseteq e_{0}(\bigvee\limits_{i\in I}\limits^{\Gamma W}H_{i})$. Choose a $s^{\ast}\in\overline{\bigcup\{H_{i}\mid i\in I\}}_{\Sigma W}$ but $s^{\ast}\not\in J$. Let $V\in\sigma(S)$ and $s^{\ast}\in V$. Then $s^{\ast}\in V\cap(S\setminus J)=V\cap W$ and by Lemma 4.1, $V\cap W\in\sigma(W)$. It follows from $s^{\ast}\in\overline{\bigcup\{H_{i}\mid i\in I\}}_{\Sigma W}$ that $$V\cap W\cap(\bigcup\{H_{i}\mid i\in I\})\neq\emptyset.$$ As a consequence, $s^{\ast}\in\overline{\bigcup\{H_{i}\mid i\in I\}}_{\Sigma S}$. Hence, $\bigvee\limits_{i\in I}\limits^{\Gamma S}e_{0}(H_{i})=e_{0}(\bigvee\limits_{i\in I}\limits^{\Gamma W}H_{i})$, proving that $e_{0}$ is Scott continuous.

It is not difficult to check that $q_{0}\circ e_{0}=id$ and $e_{0}\circ q_{0}(F)=F\cup J$ for each $F\in\Gamma S$.

Set $$\mathscr{C}=q_{0}(\mathscr{A}).$$ By Claim 1, $q_{0}(\mathscr{A})\subseteq e_{0}^{-1}(\mathscr{A})$. As $q_{0}\circ e_{0}=id$, $e_{0}^{-1}(\mathscr{A})\subseteq q_{0}(\mathscr{A})$. In conclusion, $$\mathscr{C}=q_{0}(\mathscr{A})=e_{0}^{-1}(\mathscr{A}).$$ By Claim 2, we obtain $\mathscr{C}$ is an irreducible Scott closed subset in $\Sigma\Gamma W$. Besides, $$\bigcup\mathscr{C}=\bigcup q_{0}(\mathscr{A})=\bigcup\limits_{F\in\mathscr{A}}(F\cap W)=W\cap S=W.$$

$\mathbf{Claim} \ 3$: For each $w\in W$, there exists a $C\in\mathscr{C}$ such that $C\cup\dn_{W}w\not\in\mathscr{C}$.

Assume that there is a $w^{\ast}\in W$ such that for each $C^{\ast}\in\mathscr{C}$, $C^{\ast}\cup\dn_{W}w^{\ast}\in\mathscr{C}$. Then for each $F^{\ast}\in \mathscr{A}$, $(F^{\ast}\cap W)\cup\dn_{W}w^{\ast}\in\mathscr{C}$. This yields  $$e_{0}((F^{\ast}\cap W)\cup\dn_{W}w^{\ast})=J\cup((F^{\ast}\cap W)\cup\dn_{W}w^{\ast})=J\cup F^{\ast}\cup\dn_{S}w^{\ast}\in \mathscr{A}.$$ As $\mathscr{A}$ is Scott closed in $\Gamma S$ and $F^{\ast}\cup\dn_{S}w^{\ast}\in\Gamma S$, we have $F^{\ast}\cup\dn_{S}w^{\ast}\in\mathscr{A}$. Consequently, $w^{\ast}\in J$, a contradiction. So Claim 3 holds.

Next, we consider the following two cases:

Case 1: $W=\emptyset$.

In this case, $S=J\in\mathscr{A}$. Hence, $\mathscr{A}$ contains a largest element $S$.

Case 2: $W\neq\emptyset$.

$\mathbf{Claim} \ 4$: Each element of $W$ is not compact in the subdcpo $W$.

Assume that $w_{1}$ is compact in $W$.  Then $\mathscr{H}_{w_{1}}$ is Scott open in $\Sigma\Gamma W$, where $$\mathscr{H}_{w_{1}}=\{E\in\Gamma W\mid w_{1}\in E\}=\Gamma W\setminus\dn_{\Gamma W}(W\setminus\up_{W}w_{1}).$$ As $\bigcup\mathscr{C}=W$, $\mathscr{C}\cap\mathscr{H}_{w_{1}}\neq\emptyset$. Since $\mathscr{C}$ is an irreducible closed subset in $\Sigma\Gamma W$, we conclude that $$\overline{(\mathscr{C}\cap\mathscr{H}_{w_{1}})}_{\Sigma\Gamma W}=\mathscr{C}.$$ Define a mapping $f_{w_{1}}:\Gamma W\longrightarrow\Gamma W$ by $f_{w_{1}}(E)=E\cup\dn_{W}w_{1}$. Obviously, $f_{w_{1}}$ is Scott continuous and $\mathscr{C}\cap\mathscr{H}_{w_{1}}\subseteq f_{w_{1}}^{-1}(\mathscr{C})$. So $\mathscr{C}=\overline{(\mathscr{C}\cap\mathscr{H}_{w_{1}})}_{\Sigma\Gamma W}\subseteq f_{w_{1}}^{-1}(\mathscr{C})$. This contradicts  the Claim 3. Therefore, Claim 4 is valid.

$\mathbf{Claim} \ 5$: If $W\cap P_{r_{s}}\neq\emptyset$, then $t_{s}\in W$.

Choose a $w_{2}\in W\cap P_{r_{s}}$. By Claim 4, $w_{2}$ is not a compact element in $W$. So there is a directed subset $E_{w_{2}}\subseteq W$ such that $w_{2}\leq\bigvee\limits_{W}E_{w_{2}}$ and $\up_{W}w_{2}\cap E_{w_{2}}=\emptyset$. So we get $\bigvee\limits_{W}E_{w_{2}}\not\in E_{w_{2}}$. By Lemma 4.7, $\bigvee\limits_{W}E_{w_{2}}=\bigvee\limits_{C_{0}}E_{w_{2}}=t_{s}$. As $W$ is Scott open in $S$ and $w_{2}\leq t_{s}$, we have $t_{s}\in W$.

$\mathbf{Claim} \ 6$: Let $\{F_{i}\mid i\in I\}\subseteq\Gamma W$ be a directed subset and $ F=\bigvee\limits_{i\in I}\limits^{\Gamma W}F_{i}$. If $a\in F\cap P_{r_{s}}$ and $t_{s}\not\in F$, then $a\in\bigcup\limits_{i\in I} F_{i}$.

First we assert that $a$ is compact in the subdcpo $F$. Let $D_{a}\subseteq F$ be a directed subset and $a\leq \bigvee_{F}D_{a}$. It is not difficult to find that $$\bigvee_{F}D_{a}= \bigvee_{W}D_{a}= \bigvee_{S}D_{a}= \bigvee_{C_{0}}D_{a}.$$ Suppose $\bigvee_{F}D_{a}\not\in D_{a}$. By Lemma 4.7, $\bigvee\limits_{C_{0}}D_{a}=t_{s}\in F$, impossible. Hence, $a$ is compact in $F$.

Assume that $a\not\in\bigcup\limits_{i\in I} F_{i}$. Then we have $\bigcup\limits_{i\in I} F_{i}\subseteq F\setminus\up_{F}a$. Furthermore, $F\setminus\up_{F}a$ is Scott closed in $\Sigma W$. Consequently, $F=\bigvee\limits_{i\in I}\limits^{\Gamma W}F_{i}\subseteq F\setminus\up_{F}a$, impossible. Therefore, $a\in\bigcup\limits_{i\in I} F_{i}$.

$\mathbf{Claim} \ 7$: $\{s\in\mathcal{T}_{0}\mid t_{s}\in W\}\neq\emptyset$.

Assume that for each $s\in\mathcal{T}_{0}$, $t_{s}\not\in W$.  Let $D_{W}\subseteq W$ be a directed subset. Surely, $D_{W}$ is still directed in $C_{0}$. If $\bigvee\limits_{W}D_{W}\not\in D_{W}$, then $D_{W}$ does not contain a greatest element in $C_{0}$. So we conclude that $\bigvee\limits_{W}D_{W}=t_{s}$ for some $s\in\mathcal{T}_{0}$, a contradiction. This yields that $D_{W}$ has a largest element. Equivalently, each element of $W$ is compact in $W$. This contradicts  Claim 4. Thus, Claim 7 is valid.

By Claim 7, we can fix a $t_{s}\in W$. Set $$\mathscr{O}_{s}=\{F\in\mathscr{C}\mid F\cup \dn_{W}t_{s}\not\in\mathscr{C}\}.$$

By Claim 3, $\mathscr{O}_{s}$ is nonempty. Define a mapping $f_{s}:\Gamma W\longrightarrow\Gamma W$ as follows $$\forall\ \widehat{F}\in\Gamma W,\ f_{s}(\widehat{F})=\widehat{F}\cup\dn_{W}t_{s}.$$ Similarly, it is not difficult to verify that $f_{s}$ is Scott continuous. So $f_{s}^{-1}(\mathscr{C})$ is Scott closed in $\Gamma W$. Hence, $\mathscr{\hat{O}}_{s}=\Gamma W\setminus f_{s}^{-1}(\mathscr{C})$ is Scott open in $\Gamma W$. As $\mathscr{C}$ is an irreducible Scott closed in $\Gamma W$ and $\mathscr{O}_{s}\subseteq\mathscr{C}$ is a nonempty Scott open set, we have $\overline{(\mathscr{O}_{s})}_{\Sigma\Gamma W}=\overline{(\mathscr{\hat{O}}_{s}\cap\mathscr{C})}_{\Sigma\Gamma W}=\mathscr{C}$. Clearly, $t_{s}\not\in\bigcap\mathscr{O}_{s}$.

$\mathbf{Claim} \ 8$: If $\widetilde{F}\in\mathscr{O}_{s}$, then $\widetilde{F}\cap P_{r_{s}}=\emptyset$.

Let $\widetilde{F}\in\mathscr{O}_{s}$, or equivalently, $\widetilde{F}\in\mathscr{C}$ and $t_{s}\not\in \widetilde{F}$. Assume that $x^{\ast}\in\widetilde{F}\cap P_{r_{s}}$. Set $$ \mathscr{\hat{V}}_{x^{\ast}}=\{\hat{G}\in\mathscr{O}_{s}\mid x^{\ast}\in \hat{G}\}.$$

We assert that $\hat{\mathscr{V}}_{x^*}$ is Scott open in $\mathscr{O}_{s}$. Clearly,
$\hat{\mathscr{V}}_{x^*}$ is upper in $\mathscr{O}_{s}$. Suppose $\mathscr{D}_1 \subseteq \mathscr{O}_{s}$
is directed and $\bigvee\limits_{\Gamma W} \mathscr{D}_1 \in \hat{\mathscr{V}}_{x^*}$, in other words,
$x^* \in \bigvee\limits_{\Gamma W} \mathscr{D}_1$ and $t_{s}\not\in \bigvee\limits_{\Gamma W} \mathscr{D}_1$.  By Claim~6, $x^* \in \bigcup \mathscr{D}_1$. Then $\mathscr{D}_1\cap\hat{\mathscr{V}}_{x^*}\neq\emptyset$. So $\hat{\mathscr{V}}_{x^*} \in \sigma(\mathscr{O}_{s})$.

Obviously, $\widetilde{F}\in\mathscr{\hat{V}}_{x^{\ast}}\subseteq\mathscr{O}_{s}\subseteq\mathscr{C}$. By the irreducibility of $\mathscr{C}$ again, we can conclude that $$\overline{(\mathscr{\hat{V}}_{x^{\ast}})}_{\Sigma\Gamma W}=\overline{(\mathscr{\hat{V}}_{x^{\ast}}\cap\mathscr{O}_{s})}_{\Sigma\Gamma W}=\mathscr{C}.$$
Define the mapping $f_{x^{\ast}}:\Gamma W\longrightarrow\Gamma W$ as follows $$\forall\ \widehat{G}\in\Gamma W,\ f_{x^{\ast}}(\widehat{G})=\widehat{G}\cup\dn_{W}x^{\ast}.$$ Similarly,  $f_{x^{\ast}}$ is Scott continuous. It is immediate that $\mathscr{\hat{V}}_{x^{\ast}}\cap\mathscr{O}_{s}\subseteq f_{x^{\ast}}^{-1}(\mathscr{C})$ and $f_{x^{\ast}}^{-1}(\mathscr{C})\in\Gamma(\Gamma W)$. As a result, $$\mathscr{C}=\overline{(\mathscr{\hat{V}}_{x^{\ast}}\cap\mathscr{O}_{s})}_{\Sigma\Gamma W}\subseteq f_{x^{\ast}}^{-1}(\mathscr{C}).$$ This contradicts Claim 3. Hence,  $\widetilde{F}\cap P_{r_{s}}=\emptyset$.

For each $t_{s}\in W$, define a mapping $p_{s}:W\longrightarrow W$ as follows
\[
p_{s}(w)=
\begin{cases}
t_{s}, & w\in P_{r_{s}},\\
w, & w\not\in P_{r_{s}}.
\end{cases}
\]

$\mathbf{Claim} \ 9$: $p_{s}$ is Scott continuous.

By Lemma 4.7, $t_{s}$ is an upper bound of $P_{r_{s}}$. So we obtain that $p_{s}$ is order preserving. Let $\tilde{D}\subseteq W$ be a directed subset. Without loss of generality, suppose $\bigvee\limits_{W}\tilde{D}\not\in\tilde{D}$. Then we have $\tilde{D}\subseteq P_{r_{\tilde{s}}}$ for some $\tilde{s}\in\mathcal{T}_{0}$.

Case 1: $\tilde{s}=s$.

In this case, we obtain $\bigvee\limits_{W}\tilde{D}=t_{s}$ and for each $\tilde{d}\in \tilde{D}, p_{s}(\tilde{d})=t_{s}$. So we have
 $$\bigvee\limits^{W}\limits_{\tilde{d}\in \tilde{D}}p_{s}(\tilde{d})=t_{s}=p_{s}(t_{s})= p_{s}(\bigvee\limits_{W}\tilde{D}).$$

Case 2: $\tilde{s}\neq s$.

Then we know $\bigvee\limits_{W}\tilde{D}=t_{\tilde{s}}$ and for each $d_{\star}\in \tilde{D}, p_{s}(d_{\star})=d_{\star}$. Thus

$$\bigvee\limits^{W}\limits_{d_{\star}\in \tilde{D}}p_{s}(\tilde{d})=\bigvee\limits^{W}\limits_{d_{\star}\in \tilde{D}}d_{\star}=\bigvee\limits_{W}\tilde{D}=t_{\tilde{s}}=p_{s}(t_{\tilde{s}})= p_{s}(\bigvee\limits_{W}\tilde{D}).$$

In conclusion, the Scott continuity of $p_{s}$ is obtained.

Define a mapping $c_{s}:\Gamma W\longrightarrow\Gamma W$ by $$\forall\ \hat{F}\in\Gamma W,\  c_{s}(\hat{F})=\overline{p_{s}(\hat{F})}_{\Sigma W}.$$

$\mathbf{Claim} \ 10$: $\mathscr{C}\subseteq c_{s}^{-1}(\mathscr{C})$.

By Claim 9, we deduce that $c_{s}$ is Scott continuous. Furthermore, one can check that
\[
\overline{p_{s}(\hat{F})}_{\Sigma W}=
\begin{cases}
\hat{F}\cup\dn_{W}t_{s}, &  \hat{F}\cap P_{r_{s}}\neq\emptyset,\\
\hat{F}, & \hat{F}\cap P_{r_{s}}=\emptyset.
\end{cases}
\]
 For each $\hat{O}\in\mathscr{O}_{s}$, by Claim 8, $C_{s}(\hat{O})=\hat{O}\in\mathscr{C}$. Equivalently, $\mathscr{O}_{s}\subseteq c_{s}^{-1}(\mathscr{C})$. By the Scott continuity of $c_{s}$, $c_{s}^{-1}(\mathscr{C})$ is Scott closed in $\Gamma W$. This implies $$\mathscr{C}=\overline{(\mathscr{O}_{s})}_{\Sigma\Gamma W}\subseteq c_{s}^{-1}(\mathscr{C}).$$

Set $$\widehat{D}=W\cap(\{t_{s}\mid s\in\mathcal{T}_{0}\}\cup\{q_{s,j}\mid s\in\mathcal{T}_{0},j\in\mathbb{N}\}),$$ where $q_{s,j}=(m_{s},k,j_{a,b})$.

$\mathbf{Claim} \ 11$: $\Sigma \Gamma \widehat{D}$ is sober.

First, $\widehat{D}$ is a subdcpo of $C_{0}$ because each directed subset of $W$ without a largest element is a chain contained in $\{q_{s,j}\mid j\in\mathbb{N}\}$ for some $s\in\mathcal{T}_{0}$ and $$\bigvee\limits_{W}\{q_{s,j}\mid j\in\mathbb{N}\}=\bigvee\limits_{C_{0}}\{q_{s,j}\mid j\in\mathbb{N}\}=t_{s}\in \widehat{D}.$$

Note that every non principal ideals $I\subseteq \widehat{D}$ is contained in $\{q_{s,j}\mid j\in\mathbb{N}\}$ for some $s\in\mathcal{T}_{0}$. So the collection of all non principal ideals in $\widehat{D}$ is countable. By Lemma 4.4, Lemma 4.5 and Lemma 4.6, $\Sigma \Gamma \widehat{D}$ is sober.

Let $\rho:W\longrightarrow \widehat{D}$ be the mapping defined by
\[
\rho(\tilde{w})=
\begin{cases}
t_{u}, & \exists u\in\mathcal{T}_{0}, \tilde{w}\in P_{r_{u}},\\
\tilde{w}, & \tilde{w}\in \widehat{D}.
\end{cases}
\]
The Scott continuity of $\rho$ follows from the same argument as that for the mapping $p_{s}$. As $\widehat{D}$ is a subdcpo of $W$, the inclusion mapping $\hat{i}:\widehat{D}\longrightarrow W$ is Scott continuous. Besides, $$\rho\circ\hat{i}=id_{\widehat{D}}\ \text{and}\ id_{W}\leq \hat{i}\circ \rho.$$

Let $q:\Gamma W\longrightarrow \Gamma \widehat{D}$ and $e:\Gamma \widehat{D}\longrightarrow \Gamma W$ be the mapping defined by $$q(\widetilde{F})=\overline{\rho(\widetilde{F})}_{\Sigma \widehat{D}}\ \text{and}\ e(\tilde{H})=\overline{\hat{i}(\tilde{H})}_{\Sigma W}=\overline{\tilde{H}}_{\Sigma W}.$$
Using the Scott continuity of $\rho$ and $\hat{i}$, $q$ and $e$ are Scott continuous.

$\mathbf{Claim} \ 12$: $q\circ e=id_{\Gamma \widehat{D}}$ and $id_{\Gamma W}\leq e\circ q$.

For each $\tilde{H}\in\Gamma \widehat{D}$, $$\tilde{H}=\overline{\rho(\tilde{H})}_{\Sigma \widehat{D}}\subseteq q\circ e(\tilde{H})=\overline{\rho(\overline{\tilde{H}}_{\Sigma W})}_{\Sigma \widehat{D}}.$$
Clearly, $H\subseteq \rho^{-1}(H)$ and $\rho^{-1}(H)\in\Gamma W$ because $\rho$ is Scott continuous. This leads to $\overline{\tilde{H}}_{\Sigma W}\subseteq\rho^{-1}(H)$. Then we have $$\rho(\overline{\tilde{H}}_{\Sigma W})\subseteq \rho(\rho^{-1}(\tilde{H}))\subseteq \tilde{H}.$$ Thus $\overline{\rho(\overline{\tilde{H}}_{\Sigma W})}_{\Sigma \widehat{D}}\subseteq \overline{\tilde{H}}_{\Sigma \widehat{D}}=\tilde{H}$ and hence $q\circ e(\tilde{H})=id_{\Gamma \widehat{D}}(\tilde{H})$.

For each $\widetilde{F}\in\Gamma W$ and $b\in \widetilde{F}$, we get $\rho(b)\in e\circ q(\widetilde{F})=\overline{(\overline{(\rho(\widetilde{F}))}_{\Sigma \widehat{D}})}_{\Sigma W}$. As $t_{u}$ is an upper bound of $P_{r_{u}}$, $b\leq\rho(b)$. Since $e\circ q(\widetilde{F})$ is lower in $W$, $b\in e\circ q(\widetilde{F})$. Whence, $\widetilde{F}\subseteq e\circ q(\widetilde{F})$.

$\mathbf{Claim} \ 13$: $e\circ q(\mathscr{C})\subseteq \mathscr{C}$.

For each $L\in\mathscr{C}$, set $$S(L)=\{s\in\mathcal{T}_{0}\mid L\cap P_{r_{s}}\neq\emptyset\}.$$
Let $\widetilde{E}\subseteq S(L)$ be a finite subset. Take $$L_{E}=L\cup(\bigcup\limits_{s\in E}\dn _{W}t_{s}).$$ By claim 10 and mathematical induction, $L_{E}\in\mathscr{C}$. Now, $\mathscr{S}$ is a directed subfamily of $\mathscr{C}$, where $$\mathscr{S}=\{L_{E}\mid E\subseteq S(L)\ \text{ is finite}\}.$$ It follows from $\mathscr{C}\in\Gamma(\Gamma W)$ that $$\bigvee\limits_{\Gamma W}\mathscr{S}\in\mathscr{C}.$$

$\mathbf{Subclaim}$: $\bigvee\limits_{\Gamma W}\mathscr{S}=\overline{\rho(L)}_{\Sigma W}=e\circ q(L)$.

Obviously, $\bigvee\limits_{\Gamma W}\mathscr{S}=\overline{(L\cup\{t_{s}\mid s\in S(L)\})}_{\Sigma W}$. By the definition of the mapping $\rho$, $$\rho(L)=(L\cap \widehat{D})\cup\{t_{s}\mid s\in S(L)\}.$$ It is on doubt that $\overline{\rho(L)}_{\Sigma W}\subseteq\overline{(L\cup\{t_{s}\mid s\in S(L)\})}_{\Sigma W}$. For each $l\in L$, $l\leq \rho(l)$. As $\overline{\rho(L)}_{\Sigma W}$ is lower, $l\in\overline{\rho(L)}_{\Sigma W}$. This means that $L\subseteq\overline{\rho(L)}_{\Sigma W}$. Thus $$\overline{(L\cup\{t_{s}\mid s\in S(L)\})}_{\Sigma W}=\overline{L}_{\Sigma W}\cup\overline{\{t_{s}\mid s\in S(L)\}}_{\Sigma W}\subseteq\overline{\rho(L)}_{\Sigma W}$$. The first equality sign is valid.

By the definition of $q$ and $e$, $\overline{\rho(L)}_{\Sigma W}\subseteq e\circ q(L)=\overline{(\overline{\rho(L)}_{\Sigma \widehat{D}})}_{\Sigma W}$. As $\hat{i}$ is Scott continuous, $$\hat{i}^{-1}(\overline{\rho(L)}_{\Sigma W})=\overline{\rho(L)}_{\Sigma W}\cap \widehat{D}$$ is Scott closed in $\widehat{D}$. So we can deduce that $\overline{\rho(L)}_{\Sigma \widehat{D}}\subseteq\overline{\rho(L)}_{\Sigma W}\cap \widehat{D}\subseteq\overline{\rho(L)}_{\Sigma W}$. Hence we obtain $\overline{(\overline{\rho(L)}_{\Sigma \widehat{D}})}_{\Sigma W}\subseteq\overline{\rho(L)}_{\Sigma W}$ and therefore $\overline{\rho(L)}_{\Sigma W}=e\circ q(L)$.

By this Subclaim, Claim 13 is proved.

$\mathbf{Claim} \ 14$: $S\in\mathscr{A}$.

By Claim 12 and Claim 13, $q(\mathscr{C})=e^{-1}(\mathscr{C})$. By the Scott continuity of $q$ and $e$, $q(\mathscr{C})$ is an irreducible closed subset in $\Sigma \Gamma \widehat{D}$. By Claim 11, $q(\mathscr{C})$ contains a largest element $X\in q(\mathscr{C})$. So we have $e(X)\in\mathscr{C}$. For each $\widehat{C}\in\mathscr{C}$, by Claim 12, $$\widehat{C}\subseteq e\circ q(\widehat{C})\subseteq e(X).$$ This yields that $e(X)$ is the largest element of $\mathscr{C}$. As $\bigcup\mathscr{C}=W$, $e(X)=W\in\mathscr{C}$. It follows from $\mathscr{C}=e_{0}^{-1}(\mathscr{A})$ that $$e_{0}(W)=J\cup W=S\in\mathscr{A}.$$

By Claim 14, $\mathscr{A}$ contains a greatest element $S$. Therefore, $\Sigma\Gamma C_{0}$ is sober.
\end{proof}
\begin{theorem}{\rm Let $P_{1}$  be the dcpo presented in Section 3. Then $\Sigma\Gamma P_{1}$ is sober.}
\end{theorem}

\begin{proof} (1) Let $\mathscr{F}\subseteq\Gamma P_{1}$ be an irreducible Scott closed set in $\Sigma\Gamma P_{1}$. Set $U=\bigcup\mathscr{F}$. By the proof of Subclaim 4.1 in the Lemma 4.2, $U\in\Gamma P_{1}$. We just consider the following 3 cases.

Case 1: $U=\emptyset$.

Then we can deduce that $\mathscr{F}=\{\emptyset\}$.

Case 2: $U=P_{1}$.

Choose a $F\in\mathscr{F}$ with $\top_{1}\in F$. As $F$ is lower, $F=P_{1}$. Hence, $\mathscr{F}=\Gamma P_{1}$.

Case 3: $U\neq\emptyset$ and $U\neq P_{1}$.

Let $V=U\cap B$, where $B$ is the subposet of the dcpo $P_{1}$.

$\mathbf{Claim} \ 1$: $\Sigma \Gamma(U)$ sober if and only if $\Sigma \Gamma(V)$ is sober.

We first check that the following 3 statements hold:

(i) $U$ is subdcpo of $P_{1}$;

(ii) $V$ is Scott open in subposet $U$;

(iii) $N(U)\subseteq V$.

The statements $({\rm i})$ follows by $U\in\Gamma P_{1}$. Clearly, $V$ is lower in the subposet $U$. Let $D\subseteq U$ be a directed subset with $s=\bigvee D\in V$. Assume that $D\cap V=\emptyset$. Then we obtain $$D\subseteq \mathbb{N}^{\mathbb{N}}\times\mathbb{N}.$$ Further, $$D\subseteq\{(f,n)\mid f\in\mathbb{N}^{\mathbb{N}},n\in\mathbb{N}\}$$ because $D$ is directed. If $D$ is finite, then $s$ is the largest element of $D$. Surely, $s\in D\cap V$, impossible. If $D$ is infinite, by the order of $P_{1}$, we have $s=\top_{1}\not\in U$, impossible. Consequently, $D\cap V\neq\emptyset$ and then $V\in\sigma(U)$. The statement $({\rm ii})$ is valid.

Let $D^{\ast}\subseteq U$ be a directed subset with $\bigvee D^{\ast}\not\in D^{\ast}$. As $\top_{1}\in U$ and $U\in\Gamma P_{1}$, $\bigvee D^{\ast}\in U$. Assume that $\bigvee D^{\ast}\in\mathbb{N}^{\mathbb{N}}\times\mathbb{N}$. Similarly, we can deduce that $D^{\ast}$ is finite and further contains a greatest element, impossible. As a consequence, $\bigvee D^{\ast}\in B$. The statement $({\rm iii})$ holds. By Lemma 4.2 and 3 statements, Claim 1 is proved.

Set $$L_{k}=\mathbb{N}\times\{k\}\times L \ \text{and}\ B_{\leq k}=\bigcup\limits_{l=0}\limits^{l=k}L_{k},$$ where $L$ is the subset presented in the dcpo $P_{1}$. The symbols $T_{m,k}$ denotes the maximal element $(m,k,\top)$ in the subposet $B$ of $P_{1}$.

$\mathbf{Claim} \ 2$: There exists a $n_{0}\in\mathbb{N}$ such that $$\{k\mid \exists m\in\mathbb{N},T(m,k)\in U\}\subseteq\{0,1,\cdot\cdot\cdot,n_{0}\}.$$

Set $$K=\{k\in\mathbb{N}\mid \exists m\in\mathbb{N}, T(m,k)\in U\}.$$

Assume that such $n_{0}$ does not exist. In other words, $K$ is  is unbounded. Then for each $r\in\mathbb{N}$, there is a smallest $k_{r}$ with $r<k_{r}$ such that $$\{m\in\mathbb{N}\mid T(m,k_{r})\in U\}\neq\emptyset.$$ For each $r\in \mathbb{N}$, let $$m_{k_{r}}={\rm min}\{m\in\mathbb{N}\mid T(m,k_{r})\in U\}.$$
Define a mapping $f:\mathbb{N}\longrightarrow\mathbb{N}$ by

\[
f(k)=
\begin{cases}
 m_{k_{r}}, & k\in\{k_{r}\mid r\in\mathbb{N}\},\\
 0, & otherwise.
\end{cases}
\]
Now, we have $$(f,r)<(f,k_{r})< T(f(k_{r}),k_{r})=T(m_{k_{r}},k_{r})\in U$$ in $P_{1}$. Since $U$ is Scott closed, we get $(f,r)\in U$. The arbitrariness of $r$ gives $$D_{f}=\{(f,r)\mid r\in \mathbb{N}\}\subseteq U.$$ As $U$ is Scott closed and $D_{f}$ is a chain, we have $\bigvee\limits_{P_{1}}D_{f}=\top_{1}\in U$, impossible. Consequently, $K$ is bounded and then such $n_{0}$ exists.

Let $$V_{0}=V\cap(\cup\{L_{k}\mid 0\leq k\leq n_{0}\}).$$

$\mathbf{Claim} \ 3$:  $\Sigma \Gamma V$ sober if and only if $\Sigma \Gamma V_{0}$ is sober.

Similar, we have the following 3 conclusions:

(I) $V$ is a subdcpo of $U$;

(II) $V_{0}$ is Scott open in the subposet $V$;

(III) $N(V)\subseteq V_{0}$.

Since $U$ is Scott closed in $P_{1}$ and $\top_{1}\not\in U$, the statement (I) holds.

Let $(m^{\ast}_{1},k^{\ast}_{1},z_{1})\in V_{0}$ and $(m^{\ast}_{1},k^{\ast}_{1},z_{1})\leq (m^{\ast}_{2},k^{\ast}_{2},z_{2})$ with $(m^{\ast}_{2},k^{\ast}_{2},z_{2})\in V$. Assume that $(m^{\ast}_{2},k^{\ast}_{2},z_{2})\not\in V_{0}$. Then we have $k^{\ast}_{1}\leq n_{0}<k^{\ast}_{2}$. By the order on $B$, $z_{2}=\top$. Whence $$(m^{\ast}_{2},k^{\ast}_{2},z_{2})=T_{m^{\ast}_{2},k^{\ast}_{2}}\in V\subseteq U.$$ By Claim 2, $k^{\ast}_{2}<n_{0}$, a contradiction. This yields that $V_{0}$ is upper in $V$. Suppose $D^{\star}\subseteq V$ is directed such that $\bigvee\limits_{V}D^{\star}\in V_{0}$. Choose a $(m_{d},k_{d},z_{d})\in D^{\star}$ and write $\bigvee\limits_{V}D^{\star}$ as $(m_{v},k_{v},z_{v})$. By $\bigvee\limits_{V}D^{\star}\in V_{0}$, $k_{v}\leq n_{0}$. It follows from $(m_{d},k_{d},z_{d})\leq\bigvee\limits_{V}D^{\star}=(m_{v},k_{v},z_{v})$ that $k_{d}\leq k_{v}\leq n_{0}$. Consequently, $(m_{d},k_{d},z_{d})\in V_{0}\cap D^{\star}$. Thus, the conclusion ${\rm (II)}$ is proved.

Take a directed subset $\widehat{D}\subseteq V$ with $\bigvee\limits_{V}\widehat{D}\not\in\widehat{D}$. Since $U$ is Scott closed subdcpo of $P_{1}$ and $\widehat{D}\subseteq V\subseteq U$, we obtain $\bigvee\limits_{V}\widehat{D}\in U$. Furthermore, by statement ${\rm (iii)}$ in Claim 2, $\bigvee\limits_{V}\widehat{D}=\bigvee\limits_{U}\widehat{D}\in V$. Then we find that $\widehat{D}\subseteq B$ is directed without a largest element. By the order on $B$, $\bigvee\limits_{V}\widehat{D}$ is a maximal element of $B$. So we assert that $\bigvee\limits_{V}\widehat{D}=T_{\widehat{m},\widehat{k}}\in U$. By Claim 2, $\widehat{k}\leq n_{0}$. As a result, $\bigvee\limits_{V}\widehat{D}\in V_{0}$ and then the assertion ${\rm (III)}$ is valid.

By Lemma 4.2, Claim 3 holds.

$\mathbf{Claim} \ 4$: $\Sigma \Gamma V_{0}$ is sober.

Let $\mathscr{A}\subseteq\Gamma V_{0}$ be an irreducible closed in $\Sigma \Gamma V_{0}$. Set $$S=\bigcup\mathscr{A}.$$ By the proof of Subclaim 4.1 in Lemma 4.2, $S\in\Gamma V_{0}$. So we conclude that $\mathscr{A}$ is an irreducible closed set in $\Sigma \Gamma S$. Next, we assert that $S\in\mathscr{A}$. Without loss of generality, suppose $S\neq\emptyset$.

For each $0\leq k< n_{0}$, let $$T_{>k}=S\cap(\bigcup\limits_{l=k+1}\limits^{n_{0}}L_{l})\ ,\ R_{k}=S\setminus\overline{(T_{>k})}_{\Sigma S}\ \text{and}\ O_{k}=R_{k}\cap L_{k}.$$  Then $R_{k}\subseteq B_{\leq k}$ since $T_{>k}\cap R_{k}=\emptyset$.

$\mathbf{Subclaim} \ 4.1$: $O_{k}$ is Scott open in $S$.

Let $\widetilde{S}=S\cap B_{\leq k-1}$. Clearly, $\widetilde{S}$ is lower in $S$. Suppose $\widetilde{D}$ is directed in $S\cap B_{\leq k-1}$. We just consider the case $\bigvee\limits_{S}\widetilde{D}\not\in\widetilde{D}$. This leads to $\widetilde{D}\subseteq L_{k^{\ast}}$ for some $k^{\ast}\leq k-1$ and further $\bigvee\limits_{S}\widetilde{D}={\rm max}(L_{k^{\ast}})\in B_{\leq k-1}$.  As $S$ is Scott closed in $V_{0}$, we have $$\bigvee\limits_{V_{0}}\widetilde{D}=\bigvee\limits_{S}\widetilde{D}\in S.$$ As a consequence, $\widetilde{S}$ is Scott closed in $S$. Note that $$O_{k}=(S\setminus\overline{(T_{>k})}_{\Sigma S})\cap(S\setminus\widetilde{S}).$$ Thus $O_{k}\in\sigma(S)$.

Set $$Y=\bigcup\limits_{k=0}\limits^{n_{0}}O_{k}.$$

$\mathbf{Subclaim} \ 4.2$: $Y$ is order isomorphic to $\coprod\limits_{k=0}\limits^{n_{0}}O_{k}$.

We just consider the situation that there are $x\in O_{k}$, $y\in O_{r}$ ($k\neq r$) such that $x\not\in{\rm max}(L_{k}), y\in {\rm max}(L_{r})$ and $x\leq y$. By the orders $\sqsubset_2$, $\sqsubset_3$ and $\sqsubset_4$, we obtain $k<r$. It follows that $y\in T_{>k}$. As $\overline{(T_{>k})}_{\Sigma S}$ is lower in $S$, $x\in\overline{(T_{>k})}_{\Sigma S}$. However $x\in O_{k}$ implies $x\not\in\overline{(T_{>k})}_{\Sigma S}$, impossible. Therefore, $Y\cong\coprod\limits_{k=0}\limits^{n_{0}}O_{k}$.

$\mathbf{Subclaim} \ 4.3$: $\overline{Y}_{\Sigma S}=S$.

Let $G\in\sigma(S)$ be a nonempty Scott open set. Fix $$p={\rm max}\{q\mid G\cap L_{q}\neq\emptyset, 0\leq q\leq n_{0}\}.$$
Then $G\cap L_{p}\neq\emptyset$ and $G\cap T_{>p}=\emptyset$. As $G\in\sigma(S)$ and $T_{>p}\subseteq S\setminus G$, $$\overline{(T_{>p})}_{\Sigma S}\subseteq S\setminus G.$$ It follows that $G\subseteq R_{p}$. So we have $G\cap O_{p}=G\cap L_{p}\neq\emptyset$. Consequently, $G\cap Y\neq\emptyset$ and therefore $\overline{Y}_{\Sigma S}=S$.

$\mathbf{Subclaim} \ 4.4$: $Y$ is order isomorphic to $Z=F^{\sharp}\cap U^{\sharp}$ for some $F^{\sharp}\in\Gamma C_{0}, U^{\sharp}\in\sigma(C_{0})$, where $C_{0}$ is the subdcpo presented in Lemma 4.8.

This statement obtained by proving the following facts.

$\mathbf{Fact} \ 1$: Each $O_{k}$ is the intersect of a Scott open set and a Scott closed set in $L_{k}$.

As $S$ is Scott closed in $B_{\leq n_{0}}$, $\sigma (S)$ is the subspace topology on $S$ in $\Sigma B_{\leq n_{0}}$. By Subclaim 4.3, there are $\Omega_{k}\in\sigma(B_{\leq n_{0}})$ such that $O_{k}=S\cap\Omega_{k}$. Set $$F_{k}=S\cap L_{k}\ \text{and}\ G_{k}=\Omega_{k}\cap L_{k}.$$ It follows from $O_{k}\subseteq L_{k}$ that $O_{k}=F_{k}\cap G_{k}$. Clearly, the inclusion mapping $i:L_{k}\longrightarrow B_{\leq n_{0}}$ is Scott continuous as $L_{k}$ is a subdcpo of $B_{\leq n_{0}}$. So $i^{-1}(S)=F_{k}\in\Gamma L_{k}$ and $i^{-1}(\Omega_{k})=G_{k}\in\sigma (L _{k})$. Hence Fact 1 holds.

Fix $\delta\in\triangle, r\in\mathbb{N}$ and $0\leq k\leq n_{0}$, define $$C_{k,\delta,r}=\bigcup\limits_{s\in\mathcal{T}_{0}}(\{f_{\delta}(r\cdot s)\}\times\{k\}\times L).$$ Take $$M_{0}=\mathbb{N}\setminus\bigcup\limits_{\delta\in\triangle} Im f_{\delta}.$$
For each $m\in M_{0}$, set $$\ I_{k,m}=\{m\}\times \{k\}\times L.$$
By the order on $B$, one can check that $$L_{k}\cong(\coprod\limits_{\delta\in\triangle,r\in\mathbb{N}}C_{k,\delta,r})\coprod(\coprod\limits_{m\in M_{0}}I_{k,m}).$$

Fix $r_{0}\in\mathbb{N},k_{0}\in\mathbb{N},\widetilde{\delta}=(a_{0},b_{0})\in\triangle$, suppose $$C_{0}=\{(m_{s},k_{0},z)\mid s\in\mathcal{T}_{0},z\in L\},$$ where $m_{s}=f_{\widetilde{\delta}}(r_{0}\cdot s)$.

Set $$H_{0}=\{f_{\widetilde{\delta}}(r_{0}\cdot\varepsilon)\}\times\{k_{0}\}\times L,$$ where $\varepsilon$ is the empty word (or equivalently $r_{0}\cdot\varepsilon=r_{0}$). Then $H_{0}=\dn_{C_{0}}(f_{\widetilde{\delta}}(r_{0}\cdot\varepsilon),k_{0},\top)$ is Scott closed in $C_{0}$.

$\mathbf{Fact} \ 2$: $C_{k,\delta,r}\cong C_{0}$ and $I_{k,m}\cong H_{0}$.

Let $\pi:\triangle\longrightarrow\bigtriangleup$ be the mapping defined by
\[
\pi(\gamma)=
\begin{cases}
\delta, & \gamma=\delta,\\
\widetilde{\delta}, & \gamma=\widetilde{\delta},\\
\gamma, & \gamma\not\in\{\delta,\widetilde{\delta}\}.
\end{cases}
\]
Note that $L=(\triangle\times M)\cup\{\top\}$. For each $\gamma\in\triangle$ and $u\in M$, we write $u_{\gamma}$ for $(u,\gamma)$. Define a mapping $\widehat{\pi}:L\longrightarrow L$ by
\[
\widehat{\pi}(l)=
\begin{cases}
\top, & l=\top,\\
u_{\pi(\gamma)}, & l=u_{\gamma}.
\end{cases}
\]
Then we can check that $\widehat{\pi}$ is an order isomorphism. Now, we give the order isomorphism $\theta_{\delta,k,r}$ between $C_{k,\delta,r}$ and $C_{0}$. Specifically, $\theta_{\delta,k,r}$ sends each $(f_{\delta}(r\cdot s),k,z)$ to $(f_{\widetilde{\delta}}(r_{0}\cdot s),k_{0},\widehat{\pi}(z))$.

Define the mapping $\theta_{k,m}:I_{k,m}\longrightarrow H_{0}$ by $$\theta_{k,m}(m,k,z)=(f_{\widetilde{\delta}}(r_{0}\cdot\varepsilon),k_{0},z).$$ By the order on $B$, $\theta_{k,m}$ forms an order isomorphism. Fact 2 is proved.

For each $0\leq k\leq n_{0},$ let $$\mathscr{C}_{k}=\{C_{k,\delta,r}\mid \delta\in\triangle,r\in\mathbb{N}\}\cup\{I_{k,m}\mid m\in M_{0}\}$$ and $$I=\{(k,C)\mid 0\leq k\leq n_{0},C\in\mathscr{C}_{k},O_{k}\cap C\neq\emptyset\}.$$ For each $i=(k,C)\in I$, we set $k(i)=k$
\[
C_{i}=
\begin{cases}
C_{k(i),\delta_{i},r_{i}}, & C=C_{k(i),\delta_{i},r_{i}}\\
I_{k(i),m_{i}}, & C=I_{k(i),m_{i}}
\end{cases}
\]
and further take $Y_{i}=O_{k(i)}\cap C_{i}$. Subclaim 4.3 and the order-disjoint decomposition of $L_{k}$ give that $$Y\cong\coprod\limits_{i\in I}Y_{i}.$$ According to Fact 1, $O_{k}=F_{k}\cap G_{k}$ and $F_{k}\in \Gamma L_{k},G_{k}\in \sigma(L_{k})$. For each $i\in I$, we have $$Y_{i}=O_{k_{i}}\cap C_{i}=(F_{k_{i}}\cap G_{k_{i}})\cap C_{i}=(F_{k_{i}}\cap C_{i})\cap(G_{k_{i}}\cap C_{i}).$$ Furthermore, we can conclude that $F_{k_{i}}\cap C_{i}\in\Gamma C_{i}$ and $G_{k_{i}}\cap C_{i}\in\sigma(C_{i})$. For each index $i\in I$, fix an order isomorphism
$$
\theta_i:C_i\longrightarrow H_i,
\qquad
H_i\in\{C_0,H_0\},
$$
where
\[
\theta_{i}=
\begin{cases}
\theta_{k(i),\delta_{i},r_{i}}, & C=C_{k(i),\delta_{i},r_{i}}\\
\theta_{k(i),m_{i}}, & C=I_{k(i),m_{i}}.
\end{cases}
\]

Clearly, $
H_i\in\Gamma C_0$.

Define
$$ A_i=\theta_i(Y_i), B_i=\theta_i\bigl(F_{k(i)}\cap C_i\bigr),\ \text{and}\ W_i=\theta_i\bigl(G_{k(i)}\cap C_i\bigr).$$
Since an order isomorphism is a homeomorphism for the Scott
topologies, these definitions give
$$
B_i\in\Gamma H_i,
\qquad
W_i\in\sigma(H_i).
$$
Moreover, injectivity of the isomorphism imply $A_i=B_i\cap W_i$.
Put
$$
U_i=C_0\setminus(H_i\setminus W_i).
$$
Then $H_i\setminus W_i\in\Gamma H_i$ and further $H_i\setminus W_i\in\Gamma C_0, B_i\in\Gamma C_0$ as $H_i\in\Gamma C_0$. Consequently, $U_i\in\sigma(C_0)$.
Furthermore, $H_i\cap U_i=W_i,\ \text{and}\ B_i\subseteq H_i$.
Combining these observations with $A_i=B_i\cap W_i$, we obtain
$$
A_i=B_i\cap U_i, B_i\in\Gamma C_0\ \text{and}\
U_i\in\sigma(C_0).
$$
Let $$
\mathcal T_{0}=\bigcup_{\ell\geq 0}\mathbb N^\ell,
\qquad
\mathbb N^0=\{\varepsilon\}.$$
For each $s\in\mathcal T$ and $z\in L$, define
$$[s,z]=\bigl(f_{\widetilde{\delta}}((r_0)\cdot s),k_0,z\bigr),$$
and write the local tops as $t_s=[s,\top]=(f_{\widetilde{\delta}}((r_0)\cdot s),k_0,\top)$.

For each natural number $j$, define
$$Q_j=\bigl\{[(j)\cdot s,z]\mid s\in\mathcal T,\ z\in L\bigr\}$$ and $ h_j:C_0\longrightarrow Q_j$ by $ h_j([s,z])=[(j)\cdot s,z]$. One can check that this map $h_j$ is an order isomorphism.

$\mathbf{Fact} \ 3$: $Q_{j}$ is Scott open in $C_{0}$.

First, $Q_{j}$ is an upper set. Let
$D_{1}\subseteq C_0$
be a nonempty directed subset such that $\bigvee\limits_{C_0}D_{1}\in Q_j$.
It suffices to consider the condition $\bigvee\limits_{C_0}D_{1}\not\in D_{1}$. In this case, $\bigvee\limits_{C_0}D_{1}$ is a maximal element in $C_{0}$. So we have $\bigvee\limits_{C_0}D_{1}=(f_{\widetilde{\delta}}((r_0j)\cdot \tilde{s}),k_0,\top)\in Q_{j}$ for some $\tilde{s}\in\mathcal{T}_{0}$. Since $Q_{j}$ is lower in $C_{0}$, $Q_{j}$ is contained in $C_{0}$. Hence, $Q_{j}\in\sigma(C_{0})$.

Different copies $Q_{j}$ and $Q_{j^{'}}$ are disjoint and have no comparable points, or equivalently,
$$
j\neq j',
\quad
x\in Q_j,
\quad
y\in Q_{j'}
\quad\Longrightarrow\quad
x\nleq y
\ \text{and}\
y\nleq x.
\eqno
$$
Indeed, every comparison between points with nonempty
addresses preserves the first letter of the address.

If the index set is empty, then the original space is empty
and is already isomorphic to a locally closed subdcpo of
the core. We may therefore assume that the index set is
nonempty.

Since the set $I$ is at most countable, we can fix an injection $\nu:I\longrightarrow\mathbb N$.
Define
$$
Z_i=h_{\nu(i)}(A_i),
\qquad
Z=\bigcup_{i\in I}Z_i.$$
Then we obverse that $Z_i\subseteq Q_{\nu(i)}$. Thus they are pairwise disjoint and pairwise incomparable
as different copies $Q_{j}$ and $Q_{j^{'}}$ are disjoint and have no comparable points.

$\mathbf{Fact} \ 4$: $Y\cong Z$.

Define a mapping $\Phi:Y\longrightarrow Z$ by
$$\Phi(x)=h_{\nu(i)}\bigl(\theta_i(x)\bigr),\ \mbox{for}\ x\in Y_{i}.$$
Its inverse is given on each image piece by
$$
\Phi^{-1}(z)
=
\theta_i^{-1}\bigl(h_{\nu(i)}^{-1}(z)\bigr),
\qquad
z\in Z_i.
$$

For two points in the same piece, preservation and reflection
of order follow from the two component isomorphisms.
For points in different pieces, neither pair is comparable.
Consequently,
$$
x\leq_Y y
\quad\Longleftrightarrow\quad
\Phi(x)\leq_{C_0}\Phi(y),
\qquad
x,y\in Y.$$
Thus the mapping $\Phi$ is an order isomorphic.

$\mathbf{Fact} \ 5$: $Z=F^{\sharp}\cap U^{\sharp}$ for some $F^{\sharp}\in\Gamma C_{0}$ and $ U^{\sharp}\in\sigma(C_{0})$.

Define
$$
\widetilde F_i=h_{\nu(i)}(B_i),
\qquad
\widetilde U_i=h_{\nu(i)}(U_i).
$$
Then we have $\widetilde F_i\in\Gamma Q_{\nu(i)}$ and $\widetilde U_i\in\sigma(Q_{\nu(i)})$.
Furthermore,
$
Z_i
=
h_{\nu(i)}(B_i\cap U_i)
=
\widetilde F_i\cap\widetilde U_i$.

Now put
$$Q=\bigcup_{i\in I}Q_{\nu(i)},\ F=\bigcup_{i\in I}\widetilde F_i\ \text{and}\ U^{\sharp}=\bigcup_{i\in I}\widetilde U_i.$$

By Fact 3, we have $Q\in\sigma(C_0)$. It follows from $\widetilde U_i\in\sigma(Q_{\nu(i)})$ and $Q_{\nu(i)}\in\sigma(C_0)$ that $\widetilde U_i\in\sigma(C_0)$.
for every index $i\in I$. Hence $U^{\sharp}\in\sigma(C_0)$.
Notice also that $F\subseteq Q$ and $U\subseteq Q$. Besides, $$Q\setminus F=\bigcup\limits_{i\in I}\bigl(Q_{\nu(i)}\setminus\widetilde F_i\bigr)$$ and each $Q_{\nu(i)}\setminus\widetilde F_i\in\sigma(Q_{\nu(i)})\subseteq\sigma(Q)$. So we know $Q\setminus F\in\sigma(Q)$ and therefore $F\in\Gamma Q$.

Furthermore, terms belonging to different copies have
empty intersection:
$$
\widetilde F_i\cap\widetilde U_{i'}
=
\varnothing
,\
\text{whenever }i\neq i'.
$$
Using the above relations, we obtain
$$
F\cap U^{\sharp}
=
(\bigcup_{i\in I}\widetilde F_i)
\cap
(\bigcup_{i'\in I}\widetilde U_{i'})
=
\bigcup_{i,i'\in I}
\bigl(\widetilde F_i\cap\widetilde U_{i'}\bigr)
=
\bigcup_{i\in I}
\bigl(\widetilde F_i\cap\widetilde U_i\bigr)
=
\bigcup_{i\in I}Z_i
=Z.
$$

Take $F^\sharp=
\overline{F}_{\Sigma C_0}$. As $Q\in\sigma(C_{0})$, we obtain
$
F^\sharp\cap Q
=
\overline{F}_{\Sigma Q}=F$.
Now, we have
$$
F^\sharp \cap U^{\sharp}
=
(F^\sharp\cap Q)\cap U^{\sharp}
=
F\cap U^{\sharp}
=
Z.
$$
Consequently,
$$
Z=F^\sharp\cap U^{\sharp},
\qquad
F^\sharp\in\Gamma C_0,
\qquad
U^{\sharp}\in\sigma(C_0).$$
Therefore, Subclaim 4.4 is proved.

$\mathbf{Subclaim} \ 4.5$: $S\in\mathscr{A}$.

By Lemma 4.3, Lemma 4.8 and Subclaim 4.4, we have $\Sigma\Gamma Y$ is sober. Define tow mappings $r:\Gamma S\longrightarrow\Gamma Y$ and $e:\Gamma Y\longrightarrow\Gamma S$ by $$r(R)=R\cap Y\ \text{and}\ e(K)=\overline{K}_{\Sigma S}.$$ By Subclaim 4.1 and the proofs of Claim 1 and Claim 2 in Lemma 4.2, $r,e$ are Scott continuous and further $r\circ e=id_{\Gamma Y},e\circ r\leq id_{\Gamma S}$. So $\mathscr{B}=\overline{r(\mathscr{A})}_{\Sigma\Gamma Y}$ is an irreducible closed set in $\Sigma\Gamma Y$. As $\Sigma\Gamma Y$ is sober, $\mathscr{B}=\dn_{\Gamma Y} Y^{\ast}$ for some $Y^{\ast}\in \Gamma Y$. So $$Y=Y\cap S=\bigcup r(\mathscr{A})\subseteq Y^{\ast}$$ and therefore $Y=Y^{\ast}\in\mathscr{B}$. It follows from $e\circ r\leq id_{\Gamma S}$ that $e\circ r(\mathscr{A})\subseteq\mathscr{A}$, or equivalently, $r(\mathscr{A})\subseteq e^{-1}(\mathscr{A})\in\Gamma(\Gamma Y)$. This leads to $$Y\in\mathscr{B}\subseteq e^{-1}(\mathscr{A}).$$ In other words, $e(Y)\in\mathscr{A}$ and hence by Subclaim 4.3, $e(Y)=\overline{Y}_{\Sigma S}=S\in\mathscr{A}$.

Therefore, $\Sigma \Gamma V_{0}$ is sober.

$\mathbf{Claim} \ 5$: $U\in\mathscr{F}$.

By Claim 1, Claim 3 and Claim 4, $\Sigma \Gamma U$ is sober. Since $\mathscr{F}$ is an irreducible closed set in $\Sigma\Gamma P_{1}$ and $U\in\Gamma P_{1}$, $\mathscr{F}$ is also an irreducible closed set in $\Sigma\Gamma U$. The sobriety of $\Sigma\Gamma U$ implies $\mathscr{F}=\dn_{\Gamma U}U^{\ast}$ for some $U^{\ast}\in\Gamma U$. As a consequence, $$\bigcup\mathscr{F}=\bigcup\dn_{\Gamma U}U^{\ast}=U^{\ast}=U\in\mathscr{F}.$$

Thus, $\Sigma\Gamma P_{1}$ is sober.

\end{proof}

\begin{lemma}
{\rm Let $B_{\leq n}=\bigcup\limits_{k=0}\limits^{n}L_k,$ and $W\in\Gamma B_{\leq n}$.
Then $U_W=\downarrow_{P_1}W\in\Gamma P_1$ and $W\in\sigma(U_W)$. In particular, $\Sigma\Gamma W$ is sober.}

\end{lemma}

\begin{proof}
One can check that $B_{\leq n}\in\Gamma B$. As $W\in\Gamma B_{\leq n}$, we have $W\in\Gamma B$.
Consequently,
$$U_W\cap B=\{x\in B\mid\exists w\in W,\ x\leq_{P_1}w\}
=\{x\in B\mid\exists w\in W,\ x\leq_{B}w\}=W.$$
Moreover, $\top_1\notin U_W$. In addition, by the order on $P_{1}$,
$$
(h,r)\leq T(m,k)\Longleftrightarrow r\leq k,\ m=h(k).$$
Thus, we obtain $U_W=W\cup\bigl\{(h,r)\in \mathbb{N}^{\mathbb N}\times\mathbb N\mid
\exists k\in\mathbb N,
r\leq k\leq n,
T(h(k),k)\in W
\}$.
In particular,
$(h,r)\in U_W$ implies $r\leq n$ and for every indexing function $h$,
$$
U_W\cap\{(h,r):r\in\mathbb N\}
\subseteq
\{(h,0),\ldots,(h,n)\}.
$$ So we conclude that
$\dn_{P_1}U_W=U_W$.
Let $D\subseteq U_W$ be a directed set without a greatest element.
The classification
of directed subsets contain in $U_{W}$ gives
$$\exists h\in\mathbb N^{\mathbb N},\
D\subseteq\{(h,r):r\in\mathbb N\}
\ \text{or}\
D\subseteq B.
$$
The first alternative is impossible since $D$ has a largest element. This yields
$$
D\subseteq U_W\cap B=W.
$$
As $W\in\Gamma B$ and $B$ is a subdcpo of $P_{1}$, we have
$$
\bigvee_{P_1}D
=
\bigvee_{B}D
\in W
\subseteq U_W.
$$
Therefore, $U_W\in\Gamma P_1$.

Clearly, $W$ is upper in $U_{W}$. Suppose $\widetilde{D}\subseteq U_{W}$ is a directed without a greatest element and $d_*=\bigvee\limits_{U_W}\widetilde{D}\in W$. The preceding classification argument yields
$$
\widetilde{D}\subseteq U_W\cap B=W,
$$
and therefore $\widetilde{D}\cap W=\widetilde{D}\neq\varnothing$.
Hence $W\in\sigma(U_W)$.

By Theorem 4.9, $\Sigma\Gamma P_1$ is sober. By Lemma 4.3 and $U_{W}\in\Gamma P_{1}$, we obtain
$\Sigma\Gamma U_W$ is sober. Similarly, using Lemma 4.3 again and $W\in\sigma(U_{W})$, $\Sigma\Gamma W$ is sober.\end{proof}

\begin{theorem}
{\rm Let $P_{2}$ be the dcpo constructed in Section 3. Then
$\Sigma\Gamma P_2$ is sober.}
\end{theorem}

\begin{proof}
Let $\mathscr{A}\subseteq\Gamma P_2$
be an irreducible Scott closed set in $\Gamma P_2$, and put
$U=\bigcup\mathscr{A}$.
By the proof of Subclaim 4.1 in Lemma 4.2, we have $U\in\Gamma P_2$.

It suffices to consider the situation that $U\neq\emptyset$ and $\top_{2}\not\in U$. It follows from  $U\in\Gamma P_2$ that $$\Gamma U=\{F\in\Gamma P_2\mid F\subseteq U\}=\mathord{\downarrow}_{\Gamma P_2}U.$$ Take
$$
V=U\cap B,
$$ where $B=\mathbb{N}\times\mathbb{N}\times L$ is the subdcpo of $P_{2}$.

$\mathbf{Claim \ 1}$: $V\in \Gamma B, V\in\sigma(U)$ and further $$\{\bigvee_U D^{\ast}\mid
D^{\ast}\subseteq U\ \text{is directed and}\ \bigvee_U D^{\ast}\notin D^{\ast}\}
\subseteq V.$$

In fact, $V$ is lower in $B$ and furthermore for each directed $D^{\star}$ contained in $V$, $\bigvee\limits_{B}D^{\star}\in {\rm max}(B)$ if $\bigvee\limits_{B}D^{\star}\not\in D^{\star}$. The Scott closedness of $V$ is obtained.

Clearly, $V\subseteq U$ is upper. Suppose $\widetilde{D}\subseteq U$ is directed and $\bigvee\limits_{U}\widetilde{D}\in V$. Without loss of generality, we can require that $\bigvee\limits_{U}\widetilde{D}\not\in \widetilde{D}$. It follows from $\top_{2}\not\in U$ that $\bigvee\limits_{U}\widetilde{D}=\bigvee\limits_{B}\widetilde{D}\in{\rm max}(B)$. Assume that $\widetilde{D}\cap V=\emptyset$, in other words, $\widetilde{D}\subseteq\chi\times\mathbb{N}\times\mathbb{N}$. Then we deduce that $\widetilde{D}\subseteq\{(g,n,k)\mid n\in\mathbb{N}\}$ is a infinite chain. Note that the unique upper bound in $P_{2}$ is $\top_{2}$. This yields that $\bigvee\limits_{U}\widetilde{D}$ does not exist, impossible. Consequently, $\widetilde{D}\cap V\neq\emptyset$ and thus $V$ is Scott open in $U$.

Let $\widehat{D}\subseteq U$ be a directed set without a largest element. According to the above process, $\bigvee\limits_{U}\widehat{D}\in{\rm max}(B)$ and hence $\bigvee\limits_{U}\widehat{D}\in V$. Claim 1 is valid.

For each fixed level $k$, define
$$
J_k=
\{j\in\mathbb N\mid
\text{there exists }m\in E_j\text{ with }T(m,k)\in V\}.
$$

$\mathbf{Claim \ 2}$: $J_{k}$ is finite.

Assume that $J_{k}$ is infinite. Define a mapping $\hat{g}:\mathbb{N}\longrightarrow\bigcup\limits_{n\in\mathbb{N}}E_{n}$ by

\[
\hat{g}(j)=
\begin{cases}
m_{j}, &j\in J_{k}\\
e_{j}, &j\not\in J_{k},
\end{cases}
\]
where $e_{j}$ is an arbitrary element coming from $E_{j}$, for each $j\not\in J_{k}$. As $J_{k}$ is infinite and further unbounded, for every natural
number $n$, we can choose an index $j_{n}\in J_{k}$ satisfying $n<j_{n}$.
Then we obtain
$$
(\hat{g},n,k)<(\hat{g},j_{n},k)<T(m_{j_{n}},k)\in U.
$$
As $U$ is Scott closed in $P_{2}$, we have $\{(\hat{g},n,k)\mid n\in\mathbb N\}\subseteq U$.
However, $\{(\hat{g},n,k)\mid n\in\mathbb N\}$ is a chain and hence $\bigvee\limits^{P_{2}}\limits_{n\in\mathbb N}(\hat{g},n,k)=\top_{2}\in U$, impossible. Hence
Claim 2 is proved.

For an integer $2\leq N$, put $$T_N=V\cap\bigcup_{k\geq N}L_k,$$
and define
$$
K_N=
\{(a,b)\in\Delta:
\text{there exists }m\in i(a,b)\text{ with }T(m,N)\in V\}.
$$
By Claim 2 and $i$ an injection, $K_{N}$ is finite. Define a finite set of local tops
one level below by
$$
S_N=
V\cap(\{T(a,N-1)\mid(a,b)\in K_N\}\cup\{T(b,N-1)\mid(a,b)\in K_N\}),
$$
and put
$$
H_N=
\mathord{\downarrow}_V T_N
\cup
\mathord{\downarrow}_V S_N.$$

$\mathbf{Claim \ 3}$: $H_N\in\Gamma V$ and $\overline{(T_N)}_{\Sigma V}\subseteq H_N
\subseteq V\cap\bigcup\limits_{k\geq N-2}L_k$.

As $K_{N}$ is finite, $\dn_{V} S_{N}$ is Scott closed in $V$. Clearly, $H_{N}$ is lower in $V$.

Let $D\subseteq H_{N}$ be a directed set without a greatest element. Then there are
$m,\ell\in\mathbb N$ and $\delta=(c,d)\in\Delta$ such that one of the following two alternatives holds

(i) $D=\{(m,\ell,p_\delta)\mid p\in A\}$, $A\subseteq\mathbb N$ and $A$ is infinite;

(ii) $D=\{(m,\ell,u_\delta)\mid u\in E\}$, $E\subseteq\mathbb{N}^{<\omega}$ such that for each $u,v\in E$, $u\preceq v\ \text{or}\ v\preceq u$ and further $\sup\{|u|\mid u\in E\}=+\infty$, where $|u|$ is the length of the word $u$.

In both cases, $D\subseteq L_\ell$ and $\bigvee\limits_{B}D=T(m,\ell)$.
Since
$D\subseteq V$ and $V\in\Gamma B$, we obtain
$$d_*=\bigvee_VD
=
\bigvee_{B}D
=
T(m,\ell)\in V.$$ In convenience, define a mapping $\lambda:B\longrightarrow \mathbb{N}$ by $\lambda(m,k,z)=k$. There are only the following 3 cases.

$\mathbf{Case \ 1}$: $\ell\geq N$.

Then we know
$$d_*=T(m,\ell)
\in V\cap L_\ell
\subseteq T_N
\subseteq H_N.$$

$\mathbf{Case \ 2}$: $\ell\leq N-2$.

Suppose that there is $d\in D\cap\downarrow_V T_N$.
Then there is a $y\in T_N$ such that $d\leq y$.
By the order on $B$,
$$
N
\leq\lambda(y)
\leq\lambda(d)+1
=\ell+1
\leq N-1,
$$
a contradiction. Thus $D\cap\dn_V T_N=\varnothing$.
Together with $D\subseteq H_N=\dn_V T_N\cup\dn_V S_N$,
this gives $D\subseteq\dn_V S_N$.
By $\dn_V S_N\in\Gamma V$, we get $$d_*=\bigvee_VD\in\downarrow_V S_N\subseteq H_N.$$

$\mathbf{Case \ 3}$: $\ell=N-1$.

If $D\cap\downarrow_V T_N=\varnothing$, then as in Case 2,$D\subseteq\dn_V S_N$ and hence $d_*\in\dn_V S_N\subseteq H_N$.

It remains to consider $D\cap\dn_V T_N\neq\varnothing$.
Choose $d_0\in D, y\in T_N$ such that $d_0\leq y$.
By the order on $B$,
$$
N
\leq\lambda(y)
\leq\lambda(d_0)+1
=(N-1)+1
=N.
$$
This yields $\lambda(y)=N$. So by the order on $B$, we just consider the following two subcases.

$\mathbf{Case \ 3.1}$: $d_0=(m,N-1,u_\delta)$ for some $u\in\mathbb{N}^{<\omega}$.

By the order $\sqsubset_{2}$ on $B$,
$$m=c\ \text{and}\ \exists\ w\in\mathbb{N}^{<\omega}, u\preceq w, y=T(f_{c,d}(w),N).$$

It follows from $y\in T_N\subseteq V$ and $f_{c,d}(w)\in i(c,d)$ that $(c,d)\in K_N$. Whence,
$$d_*=T(m,N-1)=T(c,N-1)\in V.$$

Thus, we obtain
$$d_*=T(c,N-1)\in V\cap\{T(a,N-1)\mid(a,b)\in K_N\}\subseteq S_N\subseteq H_N.$$

$\mathbf{Case \ 3.2}$: $d_0=(m,N-1,p_\delta)$ for some $p\in\mathbb{N}$.

By the order $\sqsubset_{3}$ on $B$, $m=d$ and further $$\exists r\in\mathbb N,p\leq r,y=T(f_{c,d}((r)),N).$$

As $y\in T_N\subseteq V$ and $f_{c,d}((r))\in i(c,d)$, we know $(c,d)\in K_N$. This leads to $d_*=T(m,N-1)=T(d,N-1)\in V$. Similarly, it yields
$$d_*=T(d,N-1)\in V\cap\{T(b,N-1):(a,b)\in K_N\}\subseteq S_N\subseteq H_N.$$

Thus all cases yield $d_*=\bigvee\limits_VD\in H_N$.
Hence, $H_N\in\Gamma V$ and surely $\overline{(T_N)}_{\Sigma V}\subseteq H_N$.

Let $x\in\dn_V T_N$. Then there exists a $y\in T_N$ such that $x\leq y$. By the order on $B$,
$$\lambda(x)\geq\lambda(y)-1\geq N-1.$$
Consequently,
$$\dn_V T_N
\subseteq
V\cap\bigcup_{k\geq N-1}L_k.$$

Next, let $x\in\dn_V S_N$. We can fix a $s\in S_N$ satisfying $x\leq s$. Based on $S_N\subseteq V\cap L_{N-1}$,
we can conclude that $\lambda(s)=N-1$.
Using the order on $B$ again,
$$
N-2
=\lambda(s)-1
\leq\lambda(x)
\leq\lambda(s)
=N-1.
$$
Therefore, $\dn_V S_N\subseteq V\cap(L_{N-2}\cup L_{N-1})$. All in all,
$$H_N=\dn_V T_N\cup\dn_V S_N\subseteq V\cap\bigcup_{k\geq N-2}L_k.$$ Claim 3 is proved.

For each $2\leq N$, define $Z_N=V\setminus\overline{(T_N)}_{\Sigma V}$.

$\mathbf{Claim \ 4}$: $Z_N\in\sigma(U)$ and $\bigcup\limits_{N\geq2}Z_N=V$.

Clearly, each $Z_N$ is Scott open in $V$ and $Z_N\subseteq Z_{N+1}$. It follows from $V\in\sigma(U)$ and $Z_N\in\sigma(V)$ that $Z_N\in\sigma(U)$.
In addition, $$Z_N\subseteq V\setminus T_N\subseteq V\cap(\bigcup\limits_{r\leq N-1}L_{r}).$$ Obviously, $\bigcup\limits_{N\geq2}Z_N\subseteq V$. Suppose $v\in V$. As the different parts $L_{k}$ are disjoint, we can fix a unique $k^{\ast}$ such that $v\in V\cap L_{k^{\ast}}$. Choose a $N^{\ast}>k^{\ast}+3$. Then $v\not\in\bigcup\limits_{l\geq N^{\ast}-2}L_{l}$ and by Claim 3, $v\not\in \overline{(T_{N})}_{\Sigma V}$. Equivalently, $v\in Z_{N^{\ast}}$. Hence, $V=\bigcup\limits_{N\geq2}Z_N$.

For each $m$, set $B_{\leq m}=\bigcup\limits_{l\leq m}L_{l}$.

$\mathbf{Claim \ 5}$: For each $2\leq N$, $\Sigma\Gamma Z_N$ is sober.

Put $$C_N=V\cap B_{\leq N-1}.$$ As $V\in\Gamma B$ and $B_{\leq N-1}\in\Gamma B$, $C_N\in\Gamma B$. Further, we deduce that $C_N\in\Gamma B_{\leq N-1}$. By Lemma 4.10, $\Sigma\Gamma C_N$ is sober. It is not difficult to check that $C_N$ is a subdcpo of $V$ and $Z_N$ is upper in $C_N$. So we can conclude that $Z_N\in\sigma(C_N)$. Using Lemma 4.3, $\Sigma\Gamma Z_N$ is sober.

$\mathbf{Claim \ 6}$: $F_{0}=\overline{V}_{\Sigma U}\in\mathscr A$.

Define two mappings $q_{N}$ and $e_{N}$ by
$$
q_N:\Gamma U\longrightarrow\Gamma Z_N,
\ \ q_N(F)=F\cap Z_N,
$$
$$
e_N:\Gamma Z_N\longrightarrow\Gamma U,
\ \ e_N(G)=\overline{G}_{\Sigma U}.
$$
As $Z_{N}$ is Scott open in $U$, by the proof  of Claim 1 in Lemma 4.2, $q_{N}$ and $e_{N}$ are Scott continuous and satisfy
$$
e_N\circ q_N(F)\subseteq F,\  q_N\circ e_N=id.
$$
Since $\mathscr A$ is an irreducible closed set in $\Sigma \Gamma P_{2}$ and $U\in\Gamma P_{2}$, $\mathscr A$ is also an irreducible closed set in $\Sigma \Gamma U$. This yields that
$$\mathscr E_N=\overline{q_N(\mathscr A)}_{\Sigma\Gamma Z_{N}}$$ is irreducible in $\Sigma\Gamma Z_{N}$.
The sobriety of $\Sigma\Gamma Z_{N}$ implies that $\mathscr E_N$ contains a largest element $G_N\in\mathscr E_N$. So we have $q_N(\mathscr A)\subseteq G_N$ and therefore
$$
\bigcup q_N(\mathscr A)=\bigcup_{F\in\mathscr A}(F\cap Z_N)
=U\cap Z_N=Z_N.
$$
Consequently, $G_N=Z_N\in\mathscr E_N$.
For each $F\in\mathscr A$, $$e_N\circ q_N(F)\subseteq F\in\mathscr A$$ and using $\mathscr A\in\Gamma(\Gamma U)$, $e_N\circ q_N(F)\in\mathscr A$. Equivalently,
$q_N(\mathscr{A})\subseteq e_N^{-1}(\mathscr A)$ and meanwhile
$\mathscr E_N\subseteq e_N^{-1}(\mathscr A)$.
In particular,
$F_N=\overline{(Z_N)}_{\Sigma U}=e_N(Z_N)\in\mathscr A.$

Then all $F_{N}$ form an increasing subfamily of $\mathscr A$. By Claim 4 and $\mathscr A\in\Gamma(\Gamma U)$,
$\bigvee\limits_{2\leq N}^{\Gamma U}F_N\in\mathscr A$ and further $$\bigvee\limits_{2\leq N}^{\Gamma U}F_N\subseteq \overline{V}_{\Sigma U}=F_{0}\in\mathscr A.$$

$\mathbf{Claim \ 7}$: Each $x\in U\setminus F_{0}$ is compact in $U$.

Let $\hat{D}$ be a directed subset of $U$ with $x\leq\bigvee\limits_U \hat{D}$. Assume that $x\not\leq \hat{d}$, for each $\hat{d}\in\hat{D}$. Then $\bigvee\limits_U \hat{D}\notin \hat{D}$ and hence $\bigvee\limits_U \hat{D}\in V\subseteq F_{0}$. Together with $ F_{0}=\overline{V}_{\Sigma U}\in\Gamma U$, $x\in F_{0}$, impossible. Therefore, $x$ is compact in $U$.

Fix such a compact point $x\in U\setminus F_{0}$. Set
$$
\mathscr H_x=\{F\in\Gamma U\mid x\in F\}
.$$
Then $\mathscr H_x$ is Scott open in $\Gamma U$. As $\bigcup\mathscr{A}=U$ and $\mathscr{A}$ is irreducible in $\Sigma\Gamma U$, $\mathscr A\cap\mathcal H_x$ is nonempty and moreover $$\overline{(\mathscr A\cap\mathcal H_x)}_{\Sigma\Gamma U}=\mathscr A.$$

$\mathbf{Claim \ 8}$: $\Sigma\Gamma P_{2}$ is sober.

Define a mapping $u_{x}$ by
$$
u_x:\Gamma U\longrightarrow\Gamma U,
\ \
u_x(F)=F\cup\mathord{\dn}_U x.
$$

Let $\{F_{j}\mid j\in J\}$ be a directed subfamily of $\Gamma U$. Then we have $$u_x(\bigvee\limits_{j\in J}\limits^{\Gamma U}F_{j})=u_x(\overline{\bigcup\{F_{j}\mid j\in J\}}_{\Sigma \Gamma U})=\mathord{\dn}_U x\cup\overline{\bigcup\{F_{j}\mid j\in J\}}_{\Sigma \Gamma U}$$ and $$\bigvee\limits_{j\in J}\limits^{\Gamma U}u_x(F_{j})=\bigvee\limits_{j\in J}\limits^{\Gamma U}(F_{j}\cup\mathord{\dn}_U x)=\mathord{\dn}_U x\cup\overline{\bigcup\{F_{j}\mid j\in J\}}_{\Sigma \Gamma U}.$$ So $u_{x}$ is Scott continuous. Clearly, $\mathscr A\cap\mathcal H_x\subseteq u_{x}^{-1}(\mathscr A)\in\Gamma(\Gamma U)$. This leads to $\mathscr A=\overline{(\mathscr A\cap\mathcal H_x)}_{\Sigma\Gamma U}\subseteq u_{x}^{-1}(\mathscr A)$, equivalently, $u_{x}(\mathscr A)\subseteq\mathscr A$.

For every finite subset $K\subseteq U\setminus F_{0}$, define
$$
F_K=F_0\cup\bigcup_{k\in K}\mathord{\downarrow}_U k.
$$
By $u_{x}(\mathscr{A})\subseteq\mathscr{A}$, for each $x\in U\setminus F_{0}$ and mathematical induction, $F_K\in\mathscr A$. Take $$\mathscr{B}=\{F_K\mid K\subseteq U\setminus F_{0}\ \text{is finite}\}.$$
Then $\mathscr{B}$ is directed and contained in $\mathscr{A}$. The Scott closedness of $\mathscr{A}$ shows that $$U=\bigvee\limits_{\Gamma U}\mathscr{B}\in\mathscr A.$$
Therefore, $\Sigma\Gamma P_2$ is sober.
\end{proof}

\begin{lemma} {\rm (see \cite{EEF31}) Let $P_{1}$ and $P_{2}$ be the dcpos presented in Section 3. Then $$\Phi=\dn_{(P_{1}\times P_{2})}\{(a,a)\mid a\in {\rm max}(B)\}$$ is an irreducible closed set in $\Sigma (P_{1}\times P_{2})$.}

\end{lemma}

\begin{theorem} {\rm Let $P_{1}$ and $P_{2}$ be the dcpos presented in Section 3. Then $\Sigma(\Gamma P_{1}\times\Gamma P_{2})$ is not sober.}
\end{theorem}

\begin{proof} Let $\eta:P_{1}\times P_{2}\longrightarrow\Gamma P_{1}\times\Gamma P_{2}$ be the mapping defined by $$\eta(x,y)=(\dn_{P_{1}}x,\dn_{P_{2}}y).$$

$\mathbf{Claim} \ 1$: $\eta$ is Scott continuous.

Fix an $x_{0}\in P_{1}$ and choose a directed set $D\subseteq P_{2}$, we have $$\eta(x_{0},\bigvee\limits_{P_{2}}D)=(\dn_{P_{1}}x_{0},\dn_{P_{2}}\bigvee\limits_{P_{2}}D)$$ and
$$\bigvee\limits_{d\in D}\limits^{\Gamma P_{2}}\eta(x_{0},d)=(\dn_{P_{1}}x_{0},\bigvee\limits_{d\in D}\limits^{\Gamma P_{2}}\dn_{P_{2}} d)=(\dn_{P_{1}}x_{0},\dn_{P_{2}}\bigvee\limits_{P_{2}}D).$$ So the mapping $\eta(x_{0},\_)$ is Scott continuous. Similarly, for each $y_{0}\in P_{2}$, $\eta(\_,y_{0})$ is also Scott continuous. Therefore, $\eta$ is Scott continuous.

Set $$\mathscr{C}=\overline{\eta(\Phi)}_{\Sigma(\Gamma P_{1}\times\Gamma P_{2})}$$ and $$\mathscr{H}_{\Phi}=\{(U,V)\mid (U,V)\in\Gamma P_{1}\times\Gamma P_{2}, U\times V\subseteq\Phi\}.$$

$\mathbf{Claim} \ 2$: $\mathscr{C}\subseteq\mathscr{H}_{\Phi}$ and $\mathscr{H}_{\Phi}\in\Gamma(\Gamma P_{1}\times\Gamma P_{2})$.

Clearly, $\mathscr{H}_{\Phi}$ is lower in $\Gamma P_{1}\times\Gamma P_{2}$. Suppose $\{(U_{i},V_{i})\mid i\in I\}$ is a directed set contained in $\mathscr{H}_{\Phi}$. Set

$$S=\bigcup \{U_{i}\mid i\in I\}\ \text{and}\ \ T=\bigcup \{V_{i}\mid i\in I\}.$$

Then we have $S\times T\subseteq \Phi$ because $\{(U_{i},V_{i})\mid i\in I\}$ is directed. For each $a\in P_{1},b\in P_{2}$, the mapping $\alpha_{a}$ and $\beta_{b}$ are Scott continuous, where $\alpha_{a}:P_{2}\longrightarrow P_{1}\times P_{2}$ and $\beta_{b}:P_{1}\longrightarrow P_{1}\times P_{2}$ are given by $$\alpha_{a}(p)=(a,p)\ \text{and}\ \ \beta_{b}(q)=(q,b).$$

Choose a $y\in T$. $\beta_{y}^{-1}(\Phi)$ is a Scott closed subset in $\Gamma P_{1}$. It follows from $S\times T\subseteq \Phi$ that $S\subseteq\beta_{y}^{-1}(\Phi)$. Consequently, $\overline{S}_{\Sigma P_{1}}\subseteq\beta_{y}^{-1}(\Phi)$, or equivalently $\overline{S}_{\Sigma P_{1}}\times\{y\}\subseteq\Phi$. The arbitrariness of $y$ indicates that $\overline{S}_{\Sigma P_{1}}\times T\subseteq\Phi$.

Select an $x\in\overline{S}_{\Sigma P_{1}}$. We know that $\alpha_{x}^{-1}(\Phi)\in\Gamma P_{2}$ and $T\subseteq\alpha_{x}^{-1}(\Phi)$. Whence $\overline{T}_{\Sigma P_{2}}\subseteq\alpha_{x}^{-1}(\Phi)$. In other words, $\{x\}\times\overline{T}_{\Sigma P_{2}}\subseteq\Phi$. So we have $$\overline{S}_{\Sigma P_{1}}\times \overline{T}_{\Sigma P_{2}}=\bigcup\limits_{x\in\overline{S}_{\Sigma P_{1}}}(\{x\}\times\overline{T}_{\Sigma P_{2}})\subseteq\Phi.$$ Therefore $$\bigvee\limits_{i\in I}\limits^{\Gamma P_{1}\times \Gamma P_{2}}(U_{i},V_{i})=(\overline{S}_{\Sigma P_{1}}, \overline{T}_{\Sigma P_{2}})\in\mathscr{H}_{\Phi}$$ and hence $\mathscr{H}_{\Phi}$ is Scott closed in $\Gamma P_{1}\times \Gamma P_{2}$.

Fix $(c,d)\in\Phi$, $\eta(c,d)=(\dn_{P_{1}}c,\dn_{P_{2}}d)$. As $\Phi$ is lower and $(c,d)\in\Phi$, we obtain that $$\dn_{P_{1}}c\times\dn_{P_{2}}d=\dn_{P_{1}\times P_{2}}(c,d)\subseteq\Phi.$$ As a consequence, $\eta(c,d)=(\dn_{P_{1}}c,\dn_{P_{2}}d)\in\mathscr{H}_{\Phi}$. This leads to $\eta(\Phi)\subseteq\mathscr{H}_{\Phi}$. Hence we have $$\mathscr{C}=\overline{\eta(\Phi)}_{\Sigma(\Gamma P_{1}\times\Gamma P_{2})}\subseteq\mathscr{H}_{\Phi}.$$

$\mathbf{Claim}\ 3$: $\Sigma(\Gamma P_{1}\times\Gamma P_{2})$ is not sober.

By Claim 1 and Lemma 4.12, $\mathscr{C}$ is an irreducible closed set in $\Sigma(\Gamma P_{1}\times\Gamma P_{2})$. Assume that $(F,G)$ is the greatest element of $\mathscr{C}$.
By the order on $B$, ${\rm max}(B)$ is infinite. Choose two different maximal elements $u,v\in{\rm max}(B)$. Then $$\eta(u,u)=(\dn_{P_{1}}u,\dn_{P_{2}}u)\leq (F,G)\ \text{and}\ \eta(v,v)=(\dn_{P_{1}}v,\dn_{P_{2}}v)\leq(F,G).$$ This implies that $u\in F$ and $v\in G$. By Claim 2, $F\times G\subseteq\Phi$. So we have $(u,v)\in\Phi$. However, there is no $w\in{\rm max}(B)$ such that $u,v\in\dn_{(P_{1}\times P_{2})}w$, a contradiction. This yields that $\mathscr{C}$ does not contain a largest element. Therefore $\Sigma(\Gamma P_{1}\times\Gamma P_{2})$ is not sober.\end{proof}

\end{document}